\documentclass[oneside,11pt]{amsart}
\usepackage{amsaddr}

\usepackage[margin=3cm]{geometry}
\usepackage[utf8]{inputenc}
\usepackage{amssymb,amsfonts}
\usepackage{lmodern}
\usepackage[english]{babel}

\usepackage{mathtools}
\usepackage{stmaryrd}

\usepackage[skip=0.5ex plus 2pt minus 1pt, indent=15pt]{parskip}

\usepackage{tikz}
\usetikzlibrary{shapes,calc,positioning, graphs,graphs.standard,quotes}

\usepackage{xcolor}

\definecolor{linkblue}{RGB}{1,1,190}
\definecolor{citegreen}{RGB}{1,190,1}
\usepackage[linkcolor=linkblue
           ,citecolor=citegreen
           ,ocgcolorlinks        % This ensures that colored links are
           ,bookmarksopen=True,  % Automatically expand bookmarks in PDF
           ]{hyperref}
\usepackage[ocgcolorlinks]{ocgx2} % This fixes line breaks in OCG links

\makeatletter
  \AtBeginDocument{
    \hypersetup{
      pdftitle  = {\@title},
      pdfauthor = {\authors}     % This assumes the AMS classes, where the
    }
  }

\makeatother
\usepackage{amsthm}
\usepackage[capitalize]{cleveref}

\usepackage{microtype}

\theoremstyle{definition}
\newtheorem {definition}{Definition}[section]

\theoremstyle{plain}
\newtheorem {theorem}[definition]{Theorem}
\crefname   {theorem}{Theorem}{Theorems}
\newtheorem*{theorem*}{Theorem}
\newtheorem {lemma}[definition]{Lemma}
\newtheorem {proposition}[definition]{Proposition}

\newtheorem {corollary}[definition]{Corollary}

\theoremstyle{remark}
\newtheorem {remark}[definition]{Remark}
\newtheorem {example}[definition]{Example}
\newtheorem*{example*}{Example}
\crefname   {example}{Example}{Examples}
\usepackage{enumitem}
\setlist[enumerate,1]{label=\textup{(\arabic*)}, ref=\textup{(}\arabic*\textup{)}, leftmargin=0.75cm}
\newlist{equivenumerate}{enumerate}{1}
\setlist[equivenumerate,1]{%
  label=\textup{(\alph*)},
  ref=\textup{(}\alph*\textup{)},
  leftmargin=0.75cm
}

\newcommand{\defit}[1]{\textsf{#1}}
\DeclareMathOperator{\pic}{Pic}
\DeclareMathOperator{\cl}{cl}
\DeclareMathOperator{\minspec}{minspec}
\DeclareMathOperator{\maxspec}{maxspec}
\DeclareMathOperator{\Spec}{Spec}
\DeclareMathOperator{\Sing}{Sing}
\DeclareMathOperator{\rank}{rank}
\DeclareMathOperator{\im}{im}
\DeclareMathOperator{\Cl}{Cl}
\newcommand{\quo}{\mathbf q}
\DeclarePairedDelimiter{\card}{\lvert}{\rvert}
\newcommand{\m}{\mathfrak m}
\newcommand{\n}{\mathfrak n}
\newcommand{\fp}{\mathfrak p}
\newcommand{\fq}{\mathfrak q}
\newcommand{\fP}{\mathfrak P}
\newcommand{\fQ}{\mathfrak Q}
\newcommand{\cO}{\mathcal O}
\newcommand{\bN}{\mathbb N}
\newcommand{\bZ}{\mathbb Z}
\newcommand{\bR}{\mathbb R}
\newcommand{\cG}{\mathcal G}
\newcommand{\agg}{\mathbf A}
\DeclareMathOperator{\supp}{supp}
\newcommand{\ind}{\mathbf 1}
\newcommand{\sL}{\mathsf L}
\newcommand{\sD}{\mathsf D}
\newcommand{\cT}{\mathcal T}
\newcommand{\cB}{\mathcal B}
\newcommand{\cC}{\mathcal C}
\newcommand{\cK}{\mathcal K}
\newcommand{\bQ}{\mathbb Q}
\newcommand{\fm}{\mathfrak m}

\newcommand{\cL}{\mathcal L}
\DeclareMathOperator{\Pic}{Pic}

\renewcommand{\vec}{\mathbf}

\title{Lattices over Bass rings and graph agglomerations}
\author{Nicholas R.~Baeth}
\address{Department of Mathematics, Franklin \& Marshall College, P.~O.~Box 3003, Lancaster PA, 17604-3003, USA}
\email{nicholas.baeth@fandm.edu}

\author{Daniel Smertnig}
\address{University of Graz, Institute for Mathematics and Scientific Computing, NAWI Graz, Heinrichstrasse 36, 8010 Graz, Austria}
\email{daniel.smertnig@uni-graz.at}


\begin{document}

\begin{abstract}
    We study direct-sum decompositions of torsion-free, finitely generated modules over a (commutative) Bass ring $R$ through the factorization theory of the corresponding monoid $T(R)$.
    Results of Levy--Wiegand and Levy--Odenthal together with a study of the local case yield an explicit description of $T(R)$.
    The monoid is typically neither factorial nor cancellative.
    Nevertheless, we construct a transfer homomorphism to a monoid of graph agglomerations---a natural class of monoids serving as combinatorial models for the factorization theory of $T(R)$.
    As a consequence, the monoid $T(R)$ is transfer Krull of finite type and several finiteness results on arithmetical invariants apply.
    We also establish results on the elasticity of $T(R)$ and characterize when $T(R)$ is half-factorial.
    (Factoriality, that is, torsion-free Krull--Remak--Schmidt--Azumaya, is characterized by a theorem of Levy--Odenthal.)
    The monoids of graph agglomerations introduced here are also of independent interest.
\end{abstract}

\maketitle
\section{Introduction}

A (commutative) ring $R$ is a \defit{Bass ring} if it is noetherian, reduced (zero is the only nilpotent element), has module-finite integral closure, and every ideal of $R$ is $2$-generated \cite{levy-wiegand85}.
Bass rings arise naturally in geometry---as coordinate rings of (not necessarily irreducible) affine algebraic curves whose only singularities are double points, in number theory---for example, as quadratic orders, and in representation theory---in the form of $\bZ[G]$ with $G$ a finite abelian group of square-free order.
Bass rings have Krull dimension at most one and belong to the larger class of stable rings --- rings in which every ideal that contains a nonzerodivisor projective over its ring of endomorphisms --- see \cite{gabelli14,Olberding-16}.

Let $R$ be a Bass ring.
We are interested in direct-sum decompositions of $R$-lattices, that is torsion-free, finitely generated $R$-modules.
In the present setting, these are precisely the maximal Cohen--Macaulay modules.
To understand the representation theory of $R$, one seeks to understand (i) the indecomposable modules, and (ii) the different ways modules decompose as direct sums of indecomposable ones.
Denoting by $T(R)$ the monoid of isomorphism classes of $R$-lattices, together with operation induced by the direct sum, this means studying the factorization theory of the monoid $T(R)$.
This monoid theoretical point of view was pioneered in work of Facchini, Herbera, and Wiegand \cite{facchini-herbera00,facchini-herbera00a,wiegand01,facchini02,facchini-wiegand04}, with later extensions to countably generated modules by Herbera and Příhoda \cite{herbera-prihoda10}; see also the surveys \cite{wiegand-wiegand09,facchini12,baeth-wiegand13} and the book \cite{facchini19}.
It permits the application of techniques originating in the study of the factorization theory of integral domains.
The monoid $T(R)$ is cancellative if the Bass ring $R$ is semilocal, or more generally if $\Pic(R)$ is trivial, but is otherwise typically noncancellative (\cref{rem:noncancellative}).

In his \emph{ubiquity paper}, Bass showed that every $R$-lattice over a Bass ring is a direct sum of ideals \cite{bass63} (the rings were later named after him).
Levy--Wiegand \cite{levy-wiegand85} described $R$-lattices in terms of a \emph{genus} and a \emph{class}, represented by an element of the ideal class semigroup $\Pic(R\mid\overline R)$ of $R$. 
This semigroup is the disjoint union of the Picard groups $\Pic(S)$ of the finitely many intermediate rings $R \subseteq S \subseteq \overline R$.
If every two $R$-lattices in the same genus are isomorphic, then $R$ has \emph{finite representation type \textup{(}FRT\textup{)}} and $T(R)$ is finitely generated.
This is always the case when $R$ is semilocal.
Moreover, Bass rings always have \emph{bounded representation type \textup{(}BRT\textup{)}} \cite{cimen-wiegand-wiegand95}.

Our focus here will be on problem (ii) mentioned above.
In the best case, the decomposition of a lattice into indecomposables is unique up to order and isomorphism, that is, the $R$-lattices satisfy the Krull--Remak--Schmidt--Azumaya property (KRSA).
Equivalently, the monoid $T(R)$ is \emph{factorial}.
The results from the pivotal papers by Levy--Odenthal \cite{levy-odenthal96b,levy-odenthal96a}, applied to the special case of Bass rings, show that $T(R)$ is factorial if and only if (a) every two $R$-lattices in the same genus are isomorphic (equivalently, the group $\Pic(R)$ is trivial), and (b)  every connected component of $\Spec(R)$ contains at most one singular maximal ideal (see \cref{p:lo-factorial}).

In the vast majority of cases the monoid $T(R)$ is therefore \emph{not} factorial.
Even in the semilocal case, where statement (a) in the previous paragraph always holds, the monoid $T(R)$ can be far from factorial.
In these cases we seek to understand the decompositions through the study of \emph{arithmetical invariants} of $T(R)$, that measure the deviation from the uniqueness, respectively, describe the non-uniqueness in qualitative and quantitative ways.

These types of questions arose at first in the factorization theory of integral domains and monoids; we mention the recent surveys \cite{halter-koch08,geroldinger-zhong20}, monographs \cite{geroldinger-halterkoch06,fontana-houston-lucas13}, and proceedings \cite{anderson97,chapman05,chapman-fontana-geroldinger-olberding16}.
For  monoids of modules, this perspective was pursued in \cite{facchini-hassler-wiegand06,diracca07,baeth-luckas11,baeth-saccon12,baeth-geroldinger14,baeth-geroldinger-grynkiewicz-smertnig15} for some classes of rings; also see the survey \cite{baeth-wiegand13}.
Typically, arithmetical invariants have been studied in cancellative settings, with recent work in some noncancellative settings by Geroldinger, Fan, Kainrath, and Tringali \cite{fan-geroldinger-kainrath-tringali17,fan-tringali18}.

For instance, to the class $[M] \in T(R)$ of an $R$-lattice $M$, we associate the \defit{set of lengths} $\sL([M]) \coloneqq \{\,k \in \bN_0 : M \cong N_1 \oplus \cdots \oplus N_k \text{ with $N_i$ indecomposable} \,\}$.
The \defit{system of sets of lengths} of the monoid $T(R)$ is $\cL(T(R)) \coloneqq \{\, \sL([M]) : M \text{ an $R$-lattice} \,\}$.
The system of sets of lengths contains a great deal of information about the direct-sum decompositions of $R$-lattices.
For example, it tells us whether $T(R)$ is \defit{half-factorial}, that is the sets of lengths are all singletons.
If not, we may ask about the \defit{elasticity}, the supremum of $l/k$ such that we can find indecomposables $M_1 \oplus \cdots \oplus M_k \cong N_1 \oplus \cdots \oplus N_l$, or even a description of the structure of the sets of lengths.

To study the system of sets of lengths, we construct a \emph{transfer homomorphism} from $T(R)$ to a simpler monoid.
To the Bass ring $R$ we associate the graph $\cG_R=(V,E,r)$ of \defit{prime ideal intersections} (our graphs are finite, possibly with multiple edges, but no loops).
The graph $\cG_R$ has as its set of vertices the minimal prime ideals of $R$, and has an edge between two minimal prime ideals $\fp$ and $\fq$ for every maximal ideal containing both $\fp$ and $\fq$.
An \defit{agglomeration} on $\cG_R$ is a map $a \colon V \cup E \to \bN_0$ assigning nonnegative numbers to every edge and every vertex in such a way that $a(v) \ge a(e)$ whenever a vertex $v$ is incident with an edge $e$.
The \defit{monoid of agglomerations on $\cG_R$}, denoted $\agg(\cG_R)$, is the additive submonoid of $\bN_0^{V \cup E}$ consisting of all agglomerations on $\cG_R$.

The central result of the present paper is the following (for missing definitions see \cref{s:background}).
\begin{theorem} \label{t:main}
  For every Bass ring $R$, there exists a transfer homomorphism $\theta\colon T(R) \to \agg(\cG_R)$.
  In particular, the monoid $T(R)$ is \emph{transfer Krull of finite type}, we have $\sL([M]) = \sL(\theta([M]))$ for all $R$-lattices $M$, and $\cL(T(R)) = \cL(\agg(\cG_R))$.
\end{theorem}

Thus, monoids of graph agglomerations serve as combinatorial models for the factorizations in $T(R)$, and consequently direct-sum decompositions of $R$-lattices; many questions about the arithmetic in $T(R)$ can be answered in $\agg(\cG_R)$ instead, and in particular, this applies to all questions about sets of lengths.
The algebraic structure of $\agg(\cG_R)$ is very easy; it is a Diophantine monoid, and therefore a finitely generated reduced Krull monoid.
Krull monoids are one of the central objects in factorization theory. 
There, factorizations are typically studied by means of \emph{monoids of zero-sum sequences} and techniques from combinatorial and additive number theory; see \cite[Chapter 1]{geroldinger-rusza09} and \cite{schmid16}.
However, for monoids of graph agglomerations it turns out to be more fruitful to study the factorization theory directly.

In \cref{s:agglomerations} we therefore initiate a study of the arithmetic of monoids of graph agglomerations.
Because agglomerations with values in $\{0,1\}$ may be identified with subgraphs of $\cG_R$, an agglomeration can be viewed as a sum of subgraphs of $\cG_R$.
It is therefore not surprising that the factorization theory is strongly linked to the structure of the underlying graph; \cref{p:agg-atoms,p:davenport,exm:elasticity,thm:rho,thm:rhok,t:factorial} illustrate this.
It seems that this class of monoids has not been considered before (certainly not from the point of factorization theory) and we hope that the link between arithmetical properties and graph-theoretical properties may be of independent interest to the factorization theory community. 
We have therefore ensured that \cref{s:agglomerations} is largely self-contained, making reference only to \cref{s:background} for notation and background.

\Cref{t:main} carries with it a large number of implications: first, general finiteness and structural results about finitely generated Krull monoids are applicable to $T(R)$; second, the specific results about monoids of graph agglomerations yield corresponding quantitative results for $T(R)$.
We refer to the main result, \cref{t:main-finiteness} below, and only point out a few particular implications here.

As before $R$ is a Bass ring.
Levy--Odenthal \cite[Theorem 1.3]{levy-odenthal96b} implies that $T(R)$ is factorial if and only if $\Pic(R)$ is trivial and there is at most one singular maximal ideal in every connected component of $\Spec(R)$.
In particular, every connected component of $\cG_R$ has at most one edge.
\Cref{t:main,t:factorial} immediately yield the following.
\begin{corollary}
  The monoid $T(R)$ is half-factorial if and only if the graph $\cG_R$ is acyclic.
\end{corollary}

A further consequence is that there exist semilocal Bass rings where $T(R)$ has arbitrarily large elasticity.
Even more, for every $l \ge 0$ there exists a semilocal Bass ring over which there exist indecomposable lattices $M_1$,~$M_2$ and $N_1$, $\ldots\,$,~$N_l$ with $M_1 \oplus M_2 \cong N_1 \oplus \cdots \oplus N_l$ (\cref{t:main-finiteness}).
This is in contrast to one-dimensional \emph{local} rings with FRT, where one always has that $\rho(T(R)) \le 3/2$ by Baeth--Luckas \cite[Theorem 3.4]{baeth-luckas11}, and Prüfer rings with the $1\tfrac{1}{2}$-generator property and small zero-divisors, where the monoid of finitely generated projective modules is always half-factorial \cite[Theorem 5.1]{baeth-geroldinger-grynkiewicz-smertnig15}.

The arithmetic of ideal semigroups in stable domains was recently studied by Bashir, Geroldinger, and Reinhart \cite{bashir-geroldinger-reinhart20}.
In particular the semigroup of nonzero [invertible] ideals of $R$ is transfer Krull if and only if it is half-factorial \cite[Theorem 5.10]{bashir-geroldinger-reinhart20}.
This is in contrast to $T(R)$, where many examples of non-half-factorial transfer Krull monoids arise.

Finally, we also show that every graph appears as prime ideal intersection graph of a semilocal Bass ring (\cref{p:realization}).

The paper is structured as follows.
In \cref{s:background}, we introduce notation and key concepts from factorization theory.
In \cref{s:modules}, we first describe $T(R)$ for local Bass rings (\cref{p:indecomposables-local}, \cref{t:local-factorial}).
Combining the local information with work of Levy--Wiegand \cite{levy-wiegand85} and the package deal theorems of Levy--Odenthal \cite{levy-odenthal96b} yields a transfer homomorphism in the non-local case (\cref{t:transhom}).
The transfer homomorphism is made more explicit and simplified in \cref{p:transfer-diophantine,p:transfer-deduplicated}.
In \cref{s:agglomerations} we turn our attention to monoids of graph agglomerations, describe their algebraic structure, and study their arithmetic.
Although several of the more basic results in this section could be deduced from corresponding results of Bass rings via the transfer homomorphism (for instance, the characterization of atoms in \cref{p:agg-atoms} corresponds to the fact that indecomposable $R$-lattices are ideals of $R$), we develop them separately from first principles.
This gives a cleaner and self-contained presentation for monoids of graph agglomerations.
Finally, in \cref{s:tie-in}, we identify the codomain of the transfer homomorphism from \cref{p:transfer-deduplicated} as a monoid of agglomerations on $\cG_R$, proving \cref{t:main}.
We also prove the realization result, \cref{p:realization}.

\subsection*{Acknowledgments} Smertnig was supported by the Austrian Science Fund (FWF) project J4079-N32. Part of the research was conducted while Smertnig was visiting the University of Waterloo; he would like to thank the Department of Pure Mathematics for the hospitality received.

\section{Background and Notation} \label{s:background}

Our goal is to describe the degree to which direct-sum decompositions of lattices (finitely generated torsion-free modules) over Bass rings are nonunique. We take the approach of instead studying the arithmetic of the monoid $T(R)=\{\,[M] : M\text{ an $R$-lattice}\,\}$ of isomorphism classes of $R$-lattices, with operation induced by the direct sum: $[M]+[N]=[M\oplus N]$. Since the set of $R$-lattices is closed under isomorphisms, direct summands, and finite direct sums, we can glean a great deal of information from the monoid $T(R)$. For example, the atoms of $T(R)$ are the classes of indecomposable $R$-lattices and the ring $R$ has \emph{finite representation type \textup{(}FRT\textup{)}} (up to isomorphism, there exist only finitely many indecomposable $R$-lattices) if and only if the monoid $T(R)$ is finitely generated.

Therefore, in the remainder of this section we introduce the necessary tools from factorization theory of commutative monoids and the arithmetical invariants that we will consider as well as the means of which to transfer information from the (usually) noncancellative monoid $T(R)$ to a cancellative, in fact Krull, monoid.

\subsection*{Monoids and factorization theory}

By a \defit{monoid}, we always mean a commutative monoid; that is, a commutative semigroup with identity. If the identity is the only invertible element, the monoid is \defit{reduced}. 
\footnote{A ring is reduced if and only if $0$ is the only nilpotent element; this corresponds neither to the monoid $(R,\cdot)$ nor to the monoid $(R,+)$ being reduced, but unfortunately both terminologies are standard. We leave it to the reader to divine the correct meaning from context.}
Since we are concerned only with monoids of modules and monoids of graph agglomerations, we will stick to additive notation throughout. A monoid $(H,0,+)$ is \defit{cancellative} if an equation of the form $a+b=a+c$ implies $b=c$. A reduced monoid $H$ satisfies the weaker property of being \defit{unit-cancellative} if, whenever $a+u=a$ with $a$,~$u \in H$, then $u=0$. 

A nonzero element $a$ of a reduced monoid $H$ is an \defit{atom} (or \defit{irreducible}) if $a=b+c$ with $b$,~$c \in H$ implies $b =0$ or $c=0$.
It is \defit{prime} provided whenever $a$ is a summand of $b+c$ for some $b$,~$c \in H$, then $a$ is a summand of $b$ or of $c$.
A weaker property, an atom $a \in H$ is \defit{absolutely irreducible} if every multiple $na$ has a unique factorization, namely $na= a+ \cdots + a$. The monoid $H$ is \defit{atomic} if every nonzero element $a \in H$ can be expressed as a sum of atoms. Assume from now on that $H$ is atomic.
In this case it makes sense to study the different factorizations of $a$ and this may be accomplished by means of studying suitable arithmetical invariants. We introduce those most important to this work. For additional invariants and further details we refer to the survey by Geroldinger--Zhong \cite{geroldinger-zhong20} and the references cited therein, as well as the monograph \cite{geroldinger-halterkoch06}. For the rest of this section we assume the following standing hypotheses.

\centerline{\emph{The monoid $H$ is nonzero, reduced, and atomic.}}

Throughout, the set of nonnegative integers is denoted by $\bN_0$, and $[a,b]=\{\, x \in \bZ : a \le x \le b \,\}$ denotes a discrete interval. The \defit{set of lengths} of a non-unit $a$ is 
\[
\sL(a) \coloneqq \{\, k  \in \bN_0 : \text{there exist atoms $a_1$, $\ldots\,$,~$a_k \in H$ with $a=a_1 + \dots  + a_k$} \,\},
\]
and $\sL(u)\coloneqq\{0\}$ for any unit $u$. The \defit{system of sets of lengths} $\cL(H) \coloneqq \{\, \sL(a) : a \in H\,\}$ is one of the basic arithmetical invariants of $H$. 

As an exact description of $\cL(H)$ is usually not possible for interesting classes of monoids $H$, simpler invariants are used to describe the structure of sets of lengths. The \defit{elasticity} of a nonzero $a \in H$ is $\rho(a) \coloneqq \sup \sL(a) / \min \sL(a) \in \bQ_{\ge 1} \cup \{\infty\}$. We set $\rho(0)\coloneqq 1$ and define the \defit{elasticity of $H$} to be $\rho(H) \coloneqq \sup\{\, \rho(a) : a \in H \,\}$. 
The elasticity is \defit{accepted} if there exists $a\in H$ with $\rho(a)=\rho(H)$.
For each $k \ge 2$, the \defit{refined elasticity} is $\rho_k(H) \coloneqq \sup\{\, \sup\sL(a) : a \in H \text{ with } k \in \sL(a) \,\} \in \bN_{\ge 2} \cup \{\infty\}$. For $k\geq 2$, one has $\rho_k(H)=\sup\,\mathcal U_k(H)$ where \[ \mathcal U_k(H)=\{\, n\colon a_1+\cdots +a_k=b_1+\cdots +b_n\text{ for atoms $a_1$, $\ldots\,$,~$a_k$, $b_1$, $\ldots\,$,~$b_n$} \,\} \] is the \defit{union of sets of lengths containing $k$}. One can show
\[
\rho(H) = \sup_{k \ge 2} \rho_k(H)/k = \lim_{k\to\infty} \rho_k(H)/k.
\]
If $\sL(a)=\{\, l_1 < l_2 < \cdots \,\}$, we set $\Delta(a)\coloneqq \{\, l_{i+1}-l_i : i \ge 1 \,\}$ to be the \defit{set of distances} of $a$.
Let $\Delta(H) \coloneqq \bigcup_{a \in H} \Delta(a)$.

As a general rule, larger elasticities, and more wildly behaved length sets are indicative of highly nonunique factorization. The monoid $H$ is \defit{half-factorial} if it is atomic and $\card{\sL(a)} = 1$ for every $a \in H$.
For an atomic monoid, this is equivalent to $\Delta(H) = \emptyset$, respectively, $\rho(H) = 1$. Even when factorization in $H$ is not unique, length sets and systems of sets of lengths can be very well structured (see \cref{t:krull-finiteness} below).

Other invariants provide finer measures of how nonunique factorization can be. Here we briefly discuss two such invariants that are mentioned in subsequent sections. 
Given a reduced cancellative monoid $H$ and two factorizations $z:=a_1+\cdots +a_l+b_1+\cdots +b_m$ and $z':=a_1+\cdots +a_l+c_1+\cdots +c_n$ with each $a_i$,~$b_j$,~$c_k$ an atom of $H$ and with $b_j\not=c_k$ for any pair $j$,~$k$. The \defit{distance} between $z$ and $z'$ is $\mathsf d(z,z')=\max\{m,n\}$. Though $\mathsf d(z,z')=0$ whenever $H$ is factorial and $z$,~$z'$ are factorizations of the same element, even when $H$ is half-factorial, the distance $d(z,z')$ can be arbitrarily large. The catenary degree provides a refinement and is defined as follows. For $a\in H$, the \defit{catenary degree} $\mathsf c(a)$ of $a$ is the smallest nonnegative integer $N$ so that given any two factorizations $z$ and $z'$ of $a$, there is a chain of factorizations $z=z_0$, $z_1$, $\ldots\,$,~$z_n=z'$ such that $\mathsf d(z_{i-1}, z_i)\leq N$ for all $i\in [1,n]$. Then the catenary degree of $H$ is $\mathsf c(H)=\sup\{\,\mathsf c(a)\colon a\in H\,\}$. It can be shown that $\mathsf c(H)=0$ if and only if $H$ is factorial. If $H$ is cancellative and not factorial, then $\mathsf c(H)\geq 2$.

For a reduced monoid $H$, and elements $a$,~$b\in H$, we set $\omega(a,b)$ to be the smallest $N\in \mathbb N_0\cup\{\infty\}$ with the following property: For all $n\in \mathbb N$ with $a_1$, $\ldots\,$,~$a_n\in H$ and $a=a_1+\cdots +a_n$, if $a=b+c$ for some $c\in H$, there exists some $I\subseteq [0,N]$ and $d\in H$ with $b+d=\sum_{i\in I}a_i$. Then $\omega(H,b)\coloneqq\sup\{\,\omega(a,b) : a\in H\,\}$ and $\omega(H)\coloneqq\sup\{\,\omega(H,b) : b \text{ is an atom}\,\}$. Put more simply, if $\omega(H,b)=N$, then whenever $b$ is a summand of a sum of elements, it must be a summand of some subsum of no more than $N$ elements. It is clear that $\omega(H, b)=0$ if and only if $b$ is a unit and $\omega(H, b)=1$ if and only if $b$ is prime. Thus the $\omega$-invariant provides a measure for how far an element is from being prime and $\omega(H)$ measures how far, even a half-factorial monoid is from being factorial.

A submonoid $S \subseteq H$ is \defit{divisor-closed}, provided that for every $s \in S$ also all the summands of $s$, as considered in $H$, are already contained in $S$.
It is clear that all questions about factorizations of an element $a \in H$ may be studied in a divisor-closed submonoid containing $a$.

\subsection*{Transfer homomorphisms and Krull monoids}

Transfer homomorphisms are a key tool in studying invariants as they allow one to transfer many arithmetic results from simpler objects to the objects of interest.

\begin{definition}
    Let $(H,0_H,+)$ and $(T,0_T,+)$ be reduced atomic monoids with neutral elements $0_H$ and $0_T$, respectively.
    A surjective monoid homomorphism $\theta\colon H \to T$ is a \defit{transfer homomorphism} if it satisfies the following properties.
    \begin{enumerate}
        \item $\theta^{-1}(0_T) = 0_H$.
        \item If $\theta(a)=s+t$ for some $a\in H$ and $s$,~$t\in T$, then there exist $b$,~$c \in H$ with $a=b+c$ and such that $\theta(b)=s$ and $\theta(c)=t$.
    \end{enumerate}
\end{definition}

Since transfer homomorphism allow the lifting of factorizations from $T$ to $H$, they transfer arithmetical information from $T$ to $H$.
Specifically, we recall the following. A proof in the cancellative setting, that carries over to the more general setting, can be found in \cite[Proposition 3.2.3]{geroldinger-halterkoch06}. 

\begin{proposition}
\label{p:transfer-implication}
    Let $\theta\colon H \to T$ be a transfer homomorphism between reduced atomic monoids.
    Then $\sL(a) = \sL(\theta(a))$ for every $a \in H$.
    In particular, we have $\cL(H) = \cL(T)$ and therefore $\rho(H)=\rho(T)$, for all $k \ge 2$ also $\rho_k(H) = \rho_k(T)$ and $\Delta(H) = \Delta(T)$.
\end{proposition}
From a factorization-theoretic standpoint, the most widely studied cancellative monoids are Krull monoids.
Krull monoids have numerous equivalent characterizations (see \cite[Theorems 2.3.11 and 2.4.8]{geroldinger-halterkoch06}), we recall two.
A monoid homomorphism $\varphi\colon H \to D$ is a \defit{divisor homomorphism} if, whenever $\varphi(a)$ is a summand of $\varphi(b)$ in $D$ for $a$,~$b \in H$, then $a$ is a summand of $b$ in $H$.
A submonoid $H \subseteq D$ is \defit{saturated} if the inclusion is a divisor homomorphism; for cancellative monoids this is equivalent to $\quo(H) \cap D = H$.
Here $\quo(H) \subseteq \quo(D)$ denote the quotient groups of the respective monoids.
A \defit{divisor theory} is a divisor homomorphism $\varphi \colon H \to \bN_0^{(I)}$ such that every standard basis vector $\vec e_i$ of $\bN_0^{(I)}$ can be expressed as $\vec e_i = \min\{ \varphi(a_1), \ldots\, \varphi(a_k) \}$ with $a_1$,~$\ldots\,$,~$a_k \in H$.
A cancellative monoid $H$ is a \defit{Krull monoid} if there exists a divisor homomorphism from $H$ into some free monoid $\bN_0^{(I)}$.
Equivalently, the monoid $H$ has a divisor theory.

A Krull monoid $H$ is finitely generated if and only if this divisor theory can be taken into a finitely generated free monoid $\bN_0^n$.
In this case $n$ is the \defit{number of prime divisors}.
Finitely generated reduced Krull monoids can be characterized as Diophantine monoids, \cite[Theorem 2.7.14]{geroldinger-halterkoch06}.
To be more specific, if $H$ is a finitely generated reduced Krull monoid, then $H\cong \ker(A)\cap \mathbb N_0^t$ where $A$ is some integer-valued matrix.

If $\varphi\colon H \to D$ is a divisor theory, the \defit{divisor class group} of $H$ is
\[
   \Cl(H)\coloneqq \quo(D) / \quo(H).
\]
Because of the uniqueness of divisor theories (\cite[Theorem 2.4.7]{geroldinger-halterkoch06}), the definition of $\Cl(H)$ is independent of the choice of the divisor theory.
The class group $G \coloneqq \Cl(H)$ and its subset $G_0$ containing prime divisors (the images of the standard basis vectors) completely describe the arithmetic of $H$ (cf.~\cite[Chapter 2.5]{geroldinger-halterkoch06}).
The monoid $H$ is factorial if and only if $G$ is trivial.
Krull monoids can also be characterized in terms of monoids of zero-sum sequences over subsets of abelian groups, but for our purposes we recall only that a not necessarily cancellative, or even commutative, monoid $H$ is \defit{transfer Krull of finite type} provided that there is a (weak) transfer homomorphism $\vartheta\colon H\rightarrow \mathcal B(G_0)$ to a monoid of zero-sum sequences $\mathcal B(G_0)$ where $G_0$ is a finite subset of some  abelian group. Equivalently, there is a transfer homomorphism from $H$ to a finitely generated Krull monoid.

When studying $\mathcal B(G_0)$, many of the combinatorial considerations involve the \defit{Davenport constant} $\sD(G_0)$, which is defined to be the length of the longest sequence of elements of $G_0$ which sum to $0$ in $G$ but such that no proper subsequence sums to $0$. Other than in special situations, there is no known formula for $\sD(G_0)$, but finiteness results can still be obtained as one always has $\sD(G_0)$ is finite if $G_0$ is finite.

There has been a great deal of research devoted to the arithmetic of transfer Krull monoids (cf.~\cite{geroldinger-zhong19}) and in \cref{s:modules} we show that the monoid of isomorphism classes of lattices over a Bass ring is transfer Krull of finite type. Consequently, we can measure the degree to which direct-sum decompositions over a Bass ring are not unique by instead studying the arithmetic of certain finitely generated Krull monoids and, in particular, certain Diophantine monoids introduced in \cref{s:modules} as well as monoids of graph agglomerations in \cref{s:agglomerations}.
In these monoids, several general finiteness results hold.
Length sets are almost arithmetical multiprogressions (AAMP) and all unions of sets of lengths are almost arithmetic progressions (AAP). AAPs (resp.~AAMPs) are essentially (unions of) arithmetic progressions, possibly with some gaps at the beginning and/or end of the sequence.
Precise definitions can be found in \cite[Section 2]{geroldinger-zhong20}.

We recall the main finiteness results as discussed in \cite[Sections 3 and 5]{geroldinger-zhong20}; see also \cite[Corollary 3.4.13]{geroldinger-halterkoch06} for more.
In fact, many of these results hold for all finitely generated [cancellative] monoids.

\begin{theorem}\label{t:krull-finiteness}
    Let $H$ be a finitely generated Krull monoid.
    \begin{enumerate}
        \item\label{kf:cat-elasticities} 
        The elasticity $\rho(H) \in \bQ$ is finite and accepted, and
        \[
        \{\, \rho(a) : a \in H \,\} = \{\, q \in \bQ : 1 \le q \le \rho(H) \,\}.
        \]
        \item\label{kf:distances} The set of distances $\Delta(H)$ is finite.
        \item\label{kf:strong-structure-unions} The monoid $H$ satisfies the Strong Structure Theorem for Unions of Sets of Lengths.
        In particular, there is $M \in \bN_0$ such that, for all $k \in \bN$, the unions $\mathcal U_k(H)$ are finite AAPs with difference $\min \Delta(H)$ and bound $M$.
        \item\label{kf:structure-lengths} The monoid $H$ satisfies the Structure Theorem for Sets of Lengths: there is an $M \in \bN_0$ such that every $L \in \mathcal L(H)$ is an AAMP with difference $d \in \Delta(H)$ and bound $M$.
        \item\label{kf:omega} We have $\omega(H) < \infty$.
        \item\label{kf:cat} The set of catenary degrees $\{\, \mathsf c(a) : a \in H \,\}$ is finite. In particular $\mathsf c(H) < \infty$.
    \end{enumerate}
\end{theorem}

\begin{proof}
    These results largely are stated in the survey \cite[Theorems 3.1 and 5.5]{geroldinger-zhong20}, but we give individual references as well.

    \ref{kf:cat-elasticities} 
    That finitely generated cancellative monoids have accepted elasticities can be found in \cite[Theorem 3.1.4]{geroldinger-halterkoch06}.
    Then necessarily $\rho(H)$ is rational.
    That all rational numbers $q$ with $1 \le q \le \rho(H)$ can then be realized as an elasticity of an element is a result of Geroldinger--Zhong in \cite[Theorem 3.1]{geroldinger-zhong19} in the context of transfer Krull monoids.
    A more general result, valid for all finitely generated cancellative monoids was recently established by Zhong in \cite{zhong19}.
    
    \ref{kf:distances}, \ref{kf:cat} Finiteness of the set of catenary degrees and distances is proved by Geroldinger--Zhong in \cite[Theorem 3.1]{geroldinger-zhong20}.

    \ref{kf:strong-structure-unions} 
    What we have stated here is actually the usual version of the Structure Theorem for Unions.
    The Strong Structure Theorem for Unions also gives information about the initial and final segments of the $\mathcal U_k(\agg(\cG))$.
    It was first proved by Tringali in \cite{tringali19} and later again using a more general framework in \cite{tringali18}.
    
    \ref{kf:structure-lengths}
    This result is due to Geroldinger; a proof can be found in \cite[Theorem 4.4.11]{geroldinger-halterkoch06}.
    
    \ref{kf:omega} By Geroldinger--Hassler \cite{geroldinger-hassler-08}.
\end{proof}

\subsection*{Graphs}

By a \defit{graph} we mean a triple $\cG=(V,E,r)$ where $V$ is a finite set of vertices, where $E$ is a finite set of edges (disjoint from $V$) and where $r\colon E \to \{\, \{ v,w\} : v \ne w \in V\, \}$ is a map associating to an edge its two incident vertices.
We write $e \sim v$ if the edge $e$ is incident with the vertex $v$ (that is, $v \in r(e)$).
Observe that our definition allows for multiple edges but not for loops. 
The graph is \defit{simple} if there are no multiple edges; that is, for each pair $v$,~$w\in V$, there is at most one $e\in E$ with $r(e)=\{v,w\}$.

A subgraph consists of a subset of the edges and vertices, satisfying the condition that for every edge of the subgraph we must also take its two incident vertices. Formally, a graph $\cG'=(V',E',r)$ is a \defit{subgraph} of $\cG$ if $V' \subseteq V$, if $E' \subseteq E$, if $r(E') \subseteq \{\, \{ v,w\} : v \ne w \in V'\,\}$, and if $r'$ is the restriction of $r$ to $E'$.

For a vertex $v$ of a graph $\cG=(V,E,r)$, its \defit{degree}, written $\deg_{\cG}(v)=|\{e\in E\colon v\in r(e)\}|$ is the number of distinct edges incident with $v$. A simple graph $\cG$ with vertex set $V$ is \defit{$k$-regular} if $\deg(v)=k$ for all $v\in V$. We write $\cC_n$ for a cycle ($2$-regular) graph on $n$ vertices and $\cK_n$ for the complete ($(n-1)$-regular) graph on $n$-vertices.

\section{Module theory} \label{s:modules}

Throughout this section, let $R$ be a reduced noetherian ring with total quotient ring $K$ and set of minimal prime ideals $\minspec(R) = \{ \fp_1, \ldots, \fp_k\}$.
For convenience, we shall always assume that $R$ is nonzero and $R \ne K$ (our results generally hold, but trivially so, in these cases).
Since $R$ is reduced, the set of zerodivisors is $\fp_1 \cup \cdots \cup \fp_k$. Then $K \cong R_{\fp_1} \times \cdots \times R_{\fp_k}$, and we write $e_i\in K$ to denote the idempotent which has $1$ in the $i$-th coordinate and $0$ everywhere else. We view $R$ as embedded in $K$, and denote the integral closure of $R$ in $K$ by $\overline R$. For an $R$-module $M$, we set $\rank_{\fp_i}(M) = \dim(M \otimes_R R_{\fp_i})$. Then $\rank(M)=(\rank_{\fp_1}(M), \ldots, \rank_{\fp_k}(M))$.

A module $M_R$ is \defit{torsion-free} if the canonical  map $M \to K \otimes_R M$ is injective. Equivalently, $M$ is torsion-free if whenever $0 \ne m \in M$ and $r \in R$ with $rm=0$, then $r$ is a zerodivisor of $R$. An \defit{$R$-lattice} $M$ is a finitely generated torsion-free $R$-module. Such a lattice $M$ can always be viewed as a submodule of $K\otimes_R M$, and it makes sense to define the \defit{coefficient ring} $\cO(M) = \{\, x \in \overline R \mid xM \subseteq M \,\}$ as the largest subring of $\overline R$ acting on $M$. Letting $\varepsilon(M) = \sum\{\, e_i : e_i M = 0 \,\}$, we observe that $M$ is a faithful $R(1-\varepsilon(M))$-module.

A \defit{Bass ring} is a noetherian reduced ring with module-finite integral closure such that every ideal is generated by two elements.
Bass rings always have Krull dimension at most $1$ and they admit the following characterization.

\begin{theorem}[{\cite[Theorem 2.1]{levy-wiegand85}}] \label{tlw:bass}
    Let $R$ be a one-dimensional reduced noetherian ring and assume that $\overline R$ is finitely generated as an $R$-module.
    The following conditions are equivalent.
    \begin{equivenumerate}
        \item \label{tlw:bass:2gen} $R$ is a Bass ring. \textup{(}Every ideal is generated by $2$ elements.\textup{)}
        \item \label{tlw:bass:cyclic}$\overline R / R$ is a cyclic $R$-module.
        \item \label{tlw:bass:summand} Every faithful $R$-lattice $M$ has a direct summand isomorphic to an invertible ideal of $\cO(M)$.
        \item \label{tlw:bass:gorenstein} Every ring between $R$ and $\overline R$ is Gorenstein.
        \item \label{tlw:bass:multiplicity} $R_\m$ has multiplicity at most $2$ for every height-one maximal ideal $\m$.
    \end{equivenumerate}
\end{theorem}

As Bass showed, every $R$-lattice over a Bass ring decomposes as a direct sum of ideals. This follows from \ref{tlw:bass:summand} together with the observation that $Re$ is a Bass ring for every idempotent $e \in K$.
In particular, every indecomposable $R$-lattice is isomorphic to an ideal of $R$.
This readily yields a description of the indecomposable $R$-lattices for local Bass rings.

\begin{proposition} \label{p:indecomposables-local}
    Let $(R,\m)$ be a local Bass ring.
    Then $R$ has at most two minimal primes and $R$ has finite representation type.
    \begin{enumerate}
        \item \label{il:1} If $\card{\minspec{R}}=1$, that is, $R$ is a domain, then all indecomposable $R$-lattices have rank $1$.
        \item \label{il:2} If $\card{\minspec{R}}=2$, then the ranks of indecomposable $R$-lattices are $(1,0)$, $(0,1)$, and $(1,1)$. Moreover, up to isomorphism, there are unique $R$-lattices of ranks $(1,0)$ and $(0,1)$.
    \end{enumerate}
\end{proposition}

\begin{proof}
    We have $\dim(R)=1$.
    The ring $R$ has finite representation type since it satisfies the Drozd--Roiter conditions; see Chapter 4 of \cite{leuschke-wiegand12}. If $\dim(R/\fp_i)=0$ for some $i \in [1,k]$, then $\fp_i = \m$ is the unique maximal ideal of $R$, and hence $\dim(R)=0$.
    Thus we have $\dim(R/\fp_i)=1$ for all $i \in [1,k]$.
    Now Theorem 14.7 of \cite{matsumura89} implies $\card{\minspec(R)} \le \mu(R) \le 2$, where $\mu(R)$ is the multiplicity of $R$.
    
    \ref*{il:1} This is clear, since every indecomposable $R$-lattice is isomorphic to an ideal of $R$.

    \ref*{il:2} Since every indecomposable $R$-lattice is isomorphic to an ideal, $(1,1)$, $(1,0)$, and $(0,1)$ are the only possible ranks. They are realized by $R$, $R/\fp_1$, and $R/\fp_2$, respectively.
    
    As before, Theorem 14.7 of \cite{matsumura89} implies $\mu(R/\fp_i) \le 1$ for $i \in \{1,2\}$.
    Thus each $R/\fp_i$ is a discrete valuation ring, and $R/\fp_i$ is the unique indecomposable $R/\fp_i$-lattice.
    If $M$ and $N$ are $R$-lattices of rank $(1,0)$, then $M_{\fp_2}=0=N_{\fp_2}$, and therefore $\fp_1$ annihilates $M_{\fp_2}$ and $N_{\fp_2}$.
    Since $\fp_1 R_{\fp_1} = 0$ as well, we have $(\fp_1 M)_{\fp_i} = 0 = (\fp_1 N)_{\fp_i}$ for $i \in \{1,2\}$.
    Since $M$ and $N$ are torsion-free, this implies $\fp_1 M = 0 = \fp_1 N$.
    Thus $M$ and $N$ are $R/\fp_1$-modules of rank $1$, and hence $M \cong N$ as $R/\fp_1$-modules.
    Consequently, $M \cong N$ as $R$-modules.
 
    Analogously, $R$ has at most one $R$-lattice of rank $(0,1)$.
\end{proof}

\begin{example}
While \cref{p:indecomposables-local} guarantees that each local Bass ring $R$ with two minimal primes has exactly one lattice of rank $(1,0)$ and one of rank $(0,1)$, the ring $R$ can be taken to have arbitrarily many indecomposable lattices with rank $(1,1)$. For example, fix $n\geq 1$ and take $R=k[x,y]_{(x,y)}/(x^2-y^{2n+2})$ for any perfect field of characteristic not $2$, $3$, or $5$. Then $R$ is a local Bass ring with completion $\widehat R\cong k\llbracket x,y\rrbracket/(x^2-y^{2n+2})$. From \cite[Theorem 4.2]{Baeth-07} we know that $\widehat R$ has exactly $n+1$ indecomposable lattices of rank $(1,1)$. Since $\card{\minspec{R}}=\card{\minspec{\widehat{R}}}$ \, \cite[Theorem 6.2]{levy-odenthal96b} implies that $R$ does as well. In fact, the completion $\widehat R$ is a ring with Dynkin type ($A_{2n+1}$) and the constant rank indecomposable lattices are precisely the overrings between $\widehat R$ and its integral closure.
\end{example}

Wiegand, in \cite{wiegand01}, showed that for a local noetherian ring, the homomorphism $T(R) \to T(\widehat R)$ induced by completion is a divisor homomorphism. This provides a standard strategy for studying $T(R)$ in the local case. We use this approach in \cref{t:local-factorial} and refer the reader to \cite{baeth-geroldinger14,baeth-wiegand13} for additional background.

\begin{theorem} \label{t:local-factorial}
    If $R$ is a local Bass ring, then $T(R)$ is factorial.
    More specifically, there is an isomorphism $T(R) \cong \bN_0^t$ for some $t \ge 0$.
\end{theorem}

\begin{proof}
We first show that the $\m$-adic completion $\widehat R$ of $R$ is a local Bass ring. By standard results, the ring $\widehat R$ is a one-dimensional local noetherian ring of multiplicity at most $2$. Since $R$ is a one-dimensional reduced local noetherian ring and $\overline R$ is a finitely generated $R$-module, the ring $R$ is \emph{analytically unramified}, that is, the completion $\widehat R$ is again reduced (see \cite[Chapter 10]{matlis73} for this non-trivial fact). By Theorem 4.3.4 of \cite{huneke-swanson06}, the integral closure of $\widehat R$ is module-finite over $\widehat R$.

Since $\widehat R$ is a complete local noetherian ring, the monoid $T(\widehat R)$ is factorial by the Krull--Remak--Schmidt--Azumaya Theorem \cite[Corollary 1.10]{leuschke-wiegand12}. Moreover, by \cite{wiegand01}, the map $T(R) \to T(\widehat R)$, $[M] \mapsto [\widehat M]$ is a divisor homomorphism. A result of Levy and Odenthal \cite[Theorem 6.2]{levy-odenthal96b} describes the image of this map: an $\widehat R$-lattice $M$ is \defit{extended} if there exists an $R$-lattice $M_0$ with $M=\widehat{M_0}$. By the Levy--Odenthal result, an $\widehat R$-module $M$ is extended if $\rank_\fP(M)=\rank_\fQ(M)$ whenever $\fP$,~$\fQ \in \minspec(\widehat R)$ with $\fP \cap R = \fQ \cap R$. If $\card{\minspec R} = \card{\minspec \widehat R}$, then this implies $T(R)\cong T(\widehat R)$ and the claim follows. The only way for the equality $\card{\minspec R} = \card{\minspec \widehat R}$ 
to fail is if $\card{\minspec(R)}=1$ and $\card{\minspec \widehat R}=2$. Let $\minspec \widehat R = \{ \fP_1, \fP_2 \}$. Then $\rank \widehat R/\fP_1 = (1,0)$, $\rank \widehat R/\fP_2 = (0,1)$, and by \cref{p:indecomposables-local} these are the only indecomposable $\widehat R$-lattices of rank not equal to $(1,1)$. Let $M$ be an $R$-lattice with completion $\widehat M \cong \widehat R/\fP_1 \oplus \widehat R/\fP_2$, and let $N_1$, $\ldots\,$,~$N_n$ be $R$-lattices representing the indecomposable $\widehat R$-lattices of rank $(1,1)$. Note that $M$ is indecomposable as an $R$-module. 

If $L$ is any $R$-lattice, then
    \[
    \widehat L \cong (\widehat R/\fP_1)^{e_1} \oplus (\widehat R/\fP_2)^{e_2} \oplus \widehat N_1^{f_1} \oplus \cdots \oplus \widehat N_n^{f_n}.
    \]
    with $e_1=e_2$.
    We conclude $L \cong M^{e_1} \oplus N_1^{f_1} \oplus \cdots \oplus N_n^{f_n}$.
    Thus $T(R)$ is factorial.
\end{proof}

\begin{remark}
A similar analysis (see \cite{baeth-luckas11}) has been completed for all one-dimensional local rings with finite representation type. In this more general setting, the ranks (at each minimal prime) of indecomposable lattices are still always $0$ or $1$, but there can be up to three minimal primes and there can be multiple non-isomorphic indecomposable modules with the same non-constant rank. Consequently, the monoid $T(R)$ is often not half-factorial. However, the elasticity of the monoid $T(R)$ in this setting never exceeds $3/2$. As we will see, this is in stark contrast to what can happen for \emph{non-local} Bass rings, where elasticities can be arbitrarily large.
\end{remark}

The localization $R_\m$  of $R$ at a maximal ideal $\m$ is regular if and only if $R_\m$ is a discrete valuation ring.
If $\fp_i$ is a minimal prime, then $R_{\fp_i}$ is a field and hence regular. Thus the singular locus of $R$ is precisely
\[
\begin{split}
\Sing(R) &=  \{\, \fp \in \Spec(R) \mid R_\fp \text{ not regular} \,\}\\
&= \{\, \m \in \maxspec(R) \mid R_\m \text{ is not a DVR}\,\}.
\end{split}
\]

We recall the following basic result.

\begin{lemma} \label{l:finite-non-dvr}
    For a Bass ring $R$, the singular locus $\Sing(R)$ is finite.
\end{lemma}

\begin{proof}
    Since the integral closure $\overline R$ is finitely generated as an $R$-module, there exists a non-zerodivisor $x \in R$ in the conductor $(R: \overline R)$.
    If $x \not \in \m$, then $x$ is a unit in $R_\m$ and hence $\overline R_\m = R_\m$.
    In this case, $R_\m$ is a local noetherian integrally closed ring of dimension $1$, and thus a discrete valuation ring.
    This implies that $\widehat R_\m$ is a discrete valuation ring as well. Since $R$ is noetherian and one-dimensional, there exist only finitely many $\m \in \maxspec(R)$ with $x \in \m$.
\end{proof}

We now work towards an explicit description of $T(R)$ for a not necessarily local Bass ring. This description follows from the combination of a Package Deal Theorem by Levy and Odenthal (Theorem 2.9 of \cite{levy-odenthal96b}), that describes the possible genera of $R$-lattices, with a theorem of Levy and R.~Wiegand describing the isomorphism classes in a genus. Recall that if $M$ is an $R$-lattice, then the \defit{genus} of $M$ is the class of all $R$-lattices $N$ such that $M_\m\cong N_\m$ for all $\m\in \maxspec(R)$.

Levy and Wiegand, in Equation (3.2.4) of \cite{levy-wiegand85}, define the \defit{class} $\cl(M)$ of an $R$-lattice. Formally, this is the isomorphism class of a faithful $R$-ideal, namely the $R$-isomorphism class of the $R$-module $$\bigoplus_{n\geq 0} \Big(\bar{\bigwedge}^n M\Big)\left(\varepsilon_n(M)\right)$$ where $\bar{\bigwedge}^n(M)$ denotes the $n$th torsion-free exterior power of $M$ --- the canonical image of $\bigwedge_R^n(M)$ in $\bigwedge_K^n(KM)$ --- and where $\varepsilon_n(M)$ denotes the idempotent $\sum\big\{\, e_i \colon \rank_{\fp_i}(M)=n\, \big\}$. 
As in the Levy--Wiegand paper, by slight abuse of notation, we also use $\cl_R(M)$ to denote a representative of said class.
We rely on the reader to deduce the correct meaning from context.

By \ref{tlw:bass:gorenstein} of \cref{tlw:bass}, if $M$ and $N$ are faithful $R$-lattices, then $\cl(M)=I$ and $\cl(N)=J$ with $I$ and $J$ invertible ideals of their respective coefficient rings. With operation given by $\cl(I)\cl(J)=\cl(IJ)$, the set of $R$-isomorphism classes of faithful lattices contained in $K$ forms the \defit{ideal class semigroup} $\pic(R\mid\overline{R})$ (\cite[Section 4.3]{levy-wiegand85}). The semigroup $\Pic(R\mid\overline{R})$ is the disjoint union of the Picard group $\pic(S)$, where $S$ ranges over all overrings of $R$ (of which there are only finitely many). As such it is an \emph{inverse semigroup}. Note, also, that $IJ$ is a fractional ideal in $S\coloneqq\cO(IJ)=\cO(I)\cO(J)$ and that $IJ$ is the product in $\pic(\cO(IJ))$ of $IS$ and $JS$.
The multiplication of two elements from the ideal class semigroup $\Pic(R\mid \overline R)$ may therefore always be carried out in a suitable Picard group of an overring.
The genus of $\cl(M)$ depends only on the genus of $M$. As $\cO(\cl(M))$ is determined locally, the coefficient ring $\cO(\cl(M))$ also depends only on the genus of $M$. 

The following result of Levy and Wiegand describes faithful $R$-lattices in a given genus. 

\begin{proposition}[{\cite[Theorem 5.2]{levy-wiegand85}}] \label{p:lw-genus}
    Let $M$ be a faithful $R$-lattice.
    Then the map $M \mapsto \cl(M)$ induces a bijection between the genus of $M$ and the group $\pic(\cO(\cl(M)))$.
\end{proposition}

If $M$ is unfaithful, the result can be applied over the Bass ring $R(1-e)$ with $e=\varepsilon(M)$; see \cite[(4.3.2)]{levy-wiegand85}. We now explicitly state how classes and Picard groups behave under this reduction.

\begin{lemma} \label{l:unfaithful}
    Let $M$ be an $R$-lattice and $e=\varepsilon(M)$.
    Then
    \[
        \cl_R(M) \cong \cl_{R(1-e)}(M) \oplus Re.
    \]
    In particular,
    \begin{enumerate}
        \item \label{unfaithful:cl} $\cl_{R(1-e)}(M) = \cl_R(M) (1-e)$.
        \item \label{unfaithful:order} $\cO(\cl_{R(1-e)}(M)) = \cO(\cl_R(M))(1-e)$.
        \item \label{unfaithful:pic} $\pic(\cO(\cl_{R(1-e)}(M))) \cong \pic(\cO(\cl_R(M))(1-e))$.
    \end{enumerate}
\end{lemma}

\begin{proof}
    The first isomorphism can be deduced from the definition of the class as follows.
    Let $t_R(M)$ denote the torsion submodule, that is, the submodule of all $m \in M$ that are annihilated by a non-zerodivisor of $R$.
    Then $\overline{\bigwedge}_R^n M \cong \bigwedge^n_R M / t_R(\bigwedge^n M)$.
    The canonical map $\pi\colon R \to R(1-e)$ is a ring epimorphism with kernel $R \cap Re = \{\, r \in R : r = er \,\}$.
    Since $\ker(\pi) \subseteq \operatorname{ann}(M)$, the module $M$ is an $R(1-e)$-module with $\bigwedge_R^n M = \bigwedge_{R(1-e)}^n M$ for every $n \ge 1$.
    One easily checks $t_R(\bigwedge_R^n M) = t_{R(1-e)}(\bigwedge_{R(1-e)}^n M)$, and thus also $\overline\bigwedge_R^n M = \overline\bigwedge_{R(1-e)}^n M$ for $n \ge 1$.
    For $n=0$ we have $\overline\bigwedge_R^0 M = R$ and $\overline\bigwedge_{R(1-e)}^0 M = R(1-e)$.
    Observing $e_i M = 0$ if and only if $\rank_{\fp_i} M_{\fp_i} =0$ yields $\varepsilon^{R}_{0}(M)=e$ and $\varepsilon^{R(1-e)}_{0}(M)=0$.
    Hence $\cl_R(M) \cong \cl_{R(1-e)}(M) \oplus Re$.
    
    The other statements are consequences of the first isomorphism.
\end{proof}

Let $S=\cO(\cl_R(M))$.  Since $S \subseteq Se \times S(1-e)$, the group $\pic(\cO(\cl_{R(1-e)}(M))) \cong \pic(S(1-e))$ is an epimorphic image of $\pic(S)$.

\begin{lemma} \label{l:pic}
    \begin{enumerate}
        \item\label{pic:iso} The genus of an $R$-lattice completely determines its isomorphism class if and only if the Picard group $\Pic(R)$ is trivial.
        \item\label{pic:semilocal} If $R$ is semilocal, then the genus of an $R$-lattice completely determines its isomorphism class.
    \end{enumerate}
\end{lemma}

\begin{proof}
\ref{pic:iso}
If $R \subseteq S \subseteq \overline R$ is an overring, then there is an epimorphism $\Pic(R) \to \Pic(S)$, and hence also an epimorphism $\Pic(R) \to \Pic(Se)$ for every idempotent $e \in K$.

If $M$,~$N$ are nonisomorphic $R$-lattices in the same genus, then $\Pic(\cO(M))$ is non-trivial by \cref{p:lw-genus}. Thus also $\Pic(R)$ is non-trivial.

Conversely, if the Picard group $\Pic(R)$ is trivial, then so are all groups $\Pic(Se)$.
If $M$,~$N$ are two $R$-lattices in the same genus, then $M$ and $N$ are faithful $Se$-lattices for $S=\cO(M)$ and an idempotent $e=1-\varepsilon(M)$.
The claim follows again from \cref{p:lw-genus}.

\ref{pic:semilocal} This is well-known.
\end{proof}

\subsection{A transfer homomorphism}

We are now prepared to piece together the results of \cite{levy-odenthal96b} and \cite{levy-wiegand85} to describe $T(R)$ as a subsemigroup of $\prod_{\m \in \maxspec(R)} T(R_\m) \times \pic(R\mid\overline R)$. After accomplishing this we will be able to construct a transfer homomorphism from $T(R)$ an easier-to-understand object.

\begin{proposition}
    Let $R$ be a Bass ring.
    Define
    \[
    \varphi \colon T(R) \to \mkern-30mu \prod_{\m \in \maxspec(R)}\mkern-30mu  T(R_\m) \times \pic(R\mid\overline R)
    \]
    by $\varphi([M]) = ( ([M_\m])_{\m}, \cl(M) )$.
    Then $\varphi$ is injective and an element $(([M(\m)])_\m, g)$ is in the image of $\varphi$ if and only if
    \begin{enumerate}
        \item $M(\m) \otimes_{R_\m} K \cong M(\n) \otimes_{R_\n} K$ for all $\m$,~$\n \in \maxspec(R)$, and
        \item $g=[I \oplus Re]$ with $I$ an invertible ideal of  $\cO(\cl(M))(1-e)$, where $M$ is any module with $M_\m\cong M(\m)$ for all $\m \in \maxspec(R)$ and $e = \varepsilon(M)$.
    \end{enumerate}
    The first condition is equivalent to $\rank_{\fp_i} M(\m) = \rank_{\fp_i} M(\n)$ whenever $\fp_i \subseteq \m \cap \n$.
\end{proposition}

\begin{proof}
    We first show that $\varphi(T(R))$ is contained in the described set.
    If $M$ is an $R$-lattice then clearly $M_\m \otimes_{R_\m} K \cong M_\n \otimes_{R_\n} K$ for $\m$,~$\n \in \maxspec(R)$.
    Moreover, the faithful $R$-lattice $\cl(M)$ is invertible in $\cO(\cl(M))$ by \ref{tlw:bass:summand} of \cref{tlw:bass} and has the described form by \cref{l:unfaithful}.
    
    Suppose that $(M_\m)_\m$ is a family of $R_\m$-lattices satisfying the stated conditions.
    For $i \in [1,k]$ and a maximal ideal $\m$ with $\fp_i \subseteq \m$, let $r_i = \rank_{\fp_i} M_\m$.
    Because of the stated hypotheses, the value of $r_i$ does not depend on the choice of $\m$.
    Set $F = (Re_1)^{r_1} \oplus \cdots \oplus (Re_k)^{r_k}$.
    Then $F$ is an $R$-lattice with $\rank_{\fp_i} F = r_i$ for all $i \in [1,k]$.
    If $\m \in \maxspec(R)$ is such that $R_\m$ is a discrete valuation ring and $\fp_i$ is the unique minimal prime ideal contained in $\m$, then $F_\m \cong R_\m^{r_i} \cong M(\m)$.
    By \cref{l:finite-non-dvr}, the singular locus is finite.
    By Levy and Odenthal's Package Deal Theorem 2.9 of \cite{levy-odenthal96b}, there exists an $R$-lattice $N$ such that $N_\m \cong M(\m)$ for all $\m \in \maxspec(R)$.
   
    Letting $e = \varepsilon(N)$, the module $N$ is faithful over the Bass ring $R(1-e)$.
    By \cref{p:lw-genus}, the genus of $N$ is in bijection with $\pic(\cO(\cl_{R(1-e)}(N)))$ via the correspondence $N' \mapsto \cl_{R(1-e)}(N')$.
    The claim now follows using \cref{l:unfaithful} by choosing a suitable $R$-lattice $M$ in the genus of $N$.
\end{proof}

If $R_\m$ is a discrete valuation ring, then $M_\m$ is fully determined by its rank at the unique minimal prime ideal $\fp_i \subseteq \m$.
Thus, we may replace all but finitely many of the $T(R_\m)$ factors in the codomain of $\varphi$ by a rank vector.
This immediately yields the following description of $\varphi(T(R))$.

\begin{corollary} \label{c:tr-iso}
    The homomorphism
    \[
    \varphi(T(R)) \to \bN_0^k \times\mkern-15mu  \prod_{\m \in \Sing(R)}\mkern-15mu  T(R_\m) \times \pic(R\mid\overline R).
    \]
    given by
    \[
    \varphi([M]) = ( \rank_{\fp_1}(M), \ldots, \rank_{\fp_k}(M), ([M_\m])_{\m \in \Sing(R)}, \cl(M) ),
    \]
    is injective.
    Its image is the submonoid defined by 
    \begin{itemize}
        \item $\rank_{\fp_i}(M) = \rank_{\fp_i}(M_\m)$ for all $\m \in\maxspec(R)$ and $\fp_i \subseteq \m$, and
        \item $\cl(M) = [I \oplus Re] \in \pic(\cO(\cl(M)))$ with $e = \varepsilon(M)$, and $[I] \in \pic(\cO(\cl(M))(1-e))$. 
    \end{itemize}
\end{corollary}

Dropping the $\pic(R\mid\overline R)$-factor, we can still obtain a transfer homomorphism. 
Unless $\Pic(R)$ is trivial (for instance, when $R$ is semilocal), we obviously lose some information. Nevertheless, this vantage point is still sufficient to study many factorization theoretical invariants, in particular sets of lengths.
The main advantage is that, since the monoids $T(R_\m)$ are factorial, the resulting monoid is a Diophantine monoid, and hence a Krull monoid.

\begin{theorem}\label{t:transhom}
    For a Bass ring $R$, the homomorphism
    \[
    \psi \colon T(R) \to \bN_0^k \times \prod_{\m \in \Sing(R)} T(R_\m)
    \]
    given by $\psi([M]) = (\rank_{\fp_1}(M), \ldots, \rank_{\fp_k}(M), ([M_\m])_{\m \in \Sing(R)})$ induces a transfer homomorphism $\psi\colon T(R) \to \psi(T(R))$.
    Its image is defined by $\rank_{\fp_i}(M) = \rank_{\fp_i}(M_\m)$ for all $\m \in \Sing(R)$ and $\fp_i \subseteq \m$.
    
    Moreover, the map $\psi$ is injective if and only if $\Pic(R)$ is trivial.
\end{theorem}

\begin{proof}
    From \cref{c:tr-iso} we know that $\psi$ is an epimorphism to the described monoid.
    If $\psi([M]) = 0$, then $M=0$.
    The only thing remaining to show is that, whenever $\psi([M]) = g + h$ with $g$,~$h \in \psi(T(R))$, then there exist $R$-lattices $N$,~$L$ with $\psi([N])=g$ and $\psi([L])=h$ such that $M \cong N \oplus L$. 
    Passing to $R(1-\varepsilon(M))$, we may assume that $M$ is faithful.
    Let $N'$ and $L'$ be $R$-lattices with $\psi([N']) = g$ and $\psi([L'])=h$.
    Then $M_\m \cong N'_\m \oplus L'_\m$ for all $\m \in \maxspec(R)$.
    
    The genus, the class, and the coefficient rings are determined locally.
    Moreover, by \cite[Proposition 3.4]{levy-wiegand85}, we have $\cl(A\oplus B)=\cl(A)\cl(B)$, and so the isomorphisms $M_\m \cong N'_\m \oplus L'_\m$ imply $\cO(\cl(N')) \subseteq \cO(\cl(M))$ and $\cO(\cl(L')) \subseteq \cO(\cl(M))$.

    Let $e = \varepsilon(N')$.
    Using \cref{p:lw-genus}, together with the epimorphism $\pic(\cO(\cl(N'))) \to \pic(\cO(\cl(M)))$, we can find $N$ in the genus of $N'$ with
    \[
        \cl(M)\cO(\cl(M))(1-e)= \cl(N)\cl(L')\cO(\cl(M))(1-e).
    \]
    With $f = \varepsilon(L)$, we can similarly choose $L$ in the genus of $L'$ with
    \[
        \cl(M)\cO(\cl(M))(1-f)=\cl(N)\cl(L)\cO(\cl(M))(1-f).
    \]
    Since $M$ is faithful, we must have $ef=0$.
    Moreover, $\cl(L)f=Rf=\cl(L')f$.
    Writing $1=(1-e)f + (1-f)$, the equalities above therefore imply
    \[
    \cl(M) \cO(\cl(M)) =  \cl(N) \cl(L) \cO(\cl(M)) = \cl(N \oplus L) \cO(\cl(M)).
    \]
    Hence $M \cong N \oplus L$ by \cref{p:lw-genus}.
    
    By \cref{l:pic} the Picard group $\Pic(R)$ is trivial if and only if the genus completely determines the isomorphism class of an $R$-lattice.
    Hence, in this case, the factor $\pic(R\mid\overline R)$ can be removed in \cref{c:tr-iso}.
    Conversely, if $\Pic(R)$ is non-trivial, then there exist two non-isomorphic $R$-lattices in the same genus and $\psi$ is not injective.
\end{proof}

Levy--Odenthal gave necessary and sufficient criteria for $R$ to satisfy [torsion-free] Krull--Remak--Schmidt--Azumaya.
Their results hold for semiprime, module-finite algebras over a commutative noetherian ring of Krull dimension $1$.
We recall the special case for Bass rings.

\begin{proposition}[{Levy--Odenthal \cite{levy-odenthal96a}}] \label{p:lo-factorial}
    Let $R$ be a Bass ring.
    The monoid $T(R)$ is factorial if and only if $\Pic(R)$ is trivial and every connected component of $\Spec(R)$ contains at most one singular maximal ideal.
\end{proposition}

\begin{proof}
    Decomposing $R$ as a product of subrings, we may assume that $\Spec(R)$ is connected in the Zariski topology, that is, that $R$ is indecomposable as a ring.
    \cite[Theorem 1.3]{levy-odenthal96a} gives three conditions that are, in conjunction, equivalent to $T(R)$ being factorial.
    Condition (a) is that the genus of every $R$-lattice contains a single isomorphism class; this is equivalent to the triviality of $\Pic(R)$;
    (b) is that there is at most one non-singular maximal ideal.
    The final condition (c) is a local condition.
    It is satisfied for Bass rings because $T(R_\fm)$ is factorial by \cref{t:local-factorial}.
\end{proof}

\begin{remark}
    The previous result can also be proved without making use of the---much more general---theorem of Levy--Odenthal (but still by using their Package Deal Theorem).
    We sketch the argument, omitting details.
    First the necessity. 
    It is easy to see that $\Pic(R)$ being trivial is necessary.
    Then the transfer homomorphism $\psi$ from \cref{t:transhom} induces an isomorphism $T(R) \to \psi(T(R))$ and it suffices to study $\psi(T(R))$.
    With a bit of work, from the connectedness of $\Spec(R)$, one can deduce that there is at most one maximal ideal containing more than one minimal prime ideal.
    (In fact, this follows from the existence of the transfer homomorphism, together with \cref{t:factorial,p:transfer-bass-agg} below.)
    Each singular $R_\fm$ has at least two non-isomorphic indecomposable $R$-lattices of constant rank $1$: namely $R_\fm$ and $\overline{R_\fm}$.
    From this, it is easy to see $\card{\Sing(R)} \le 1$.
    This proves the necessity.
    
    For the sufficiency, it is again easy to see that if $\card{\Sing(R)}=1$, then $\psi(T(R))$ is factorial (recall that $T(R_\fm)\cong \bN_0^t$ by \cref{t:local-factorial}).
\end{remark}

\subsection{Construction of \texorpdfstring{$\psi(T(R))$}{psi(T(R))} as a Diophantine monoid} \label{subsec:diophantine}

In this section we further flesh out the structure of the monoid $T(R)$ when $R$ is a Bass ring. More specifically, we completely describe the image $\psi(T(R))$ under the transfer homomorphism given in \cref{t:transhom}, showing that $T(R)$ is transfer Krull. Then we state arithmetical results about $T(R)$ that are proved later in \cref{s:tie-in}.

We will describe $\psi(T(R))$ as a Diophantine monoid, and will make use of the following basic lemma showing that duplicate columns can essentially be ignored when considering a Diophantine monoid. That is, there is a natural transfer homomorphism between two Diophantine monoids; one with all duplicate columns removed. For simplicity, we state the lemma when the second column is a duplicate of the first. Induction and a permutation of columns yields the more general result.

\begin{lemma}\label{l:duplicate columns}
Let $B=\begin{bmatrix} \mathbf{a}_1 & \cdots & \mathbf{a}_n\end{bmatrix}$ be an $m\times n$ matrix with integer entries such that $\mathbf{a}_1=\mathbf{a}_2$ and take $H=\ker(B)\cap \mathbb N_0^n$. Set $B'=\begin{bmatrix} \mathbf{a}_1 & \mathbf{a}_3 & \cdots & \mathbf{a}_n\end{bmatrix}$ and $H'=\ker(B')\cap \mathbb N_0^{n-1}$. Then there is a transfer homomorphism $\theta: H\rightarrow H'$.
\end{lemma}

\begin{proof}
Define $\theta: H\rightarrow H'$ by $$\theta (x_1, x_2, x_3, \ldots, x_n)=(x_1+x_2, x_3, \ldots, x_n).$$ If $(y_1, y_2, \ldots, y_{n-1})\in H'$, then $y_1\mathbf{a}_1+\sum_{i=2}^{n-1}y_i\mathbf{a}_i=\mathbf{0}$, and so $(y_1, 0, y_3, \ldots, y_n)\in H$, whence $\theta$ is surjective. Suppose $\theta(\mathbf{x})=\mathbf{y}+\mathbf{z}$. Then $x_1+x_2=y_1+z_1$ and $y_{i-1}+z_{i-1}=x_i$ for all $i\in [3,n]$. Since either $x_1\leq y_1$ or $x_2\leq z_1$, we assume without loss of generality that $y_1\leq x_1$. Set $\mathbf{w}=(x_1-y_1, x_2, z_2, \ldots, z_{n-1})$ and $\mathbf{v}=(y_1, 0, y_2, y_3, \ldots, y_{n-1})$. Then $\theta(\mathbf{v})=\mathbf{y}$, $\theta(\mathbf{w})=\mathbf{z}$, and $\mathbf{x}=\mathbf{v}+\mathbf{w}.$ Thus $\theta$ is a transfer homomorphism.
\end{proof}

The transfer homomorphism
\[
\psi \colon T(R) \to \bN_0^k \times \prod_{\m \in \Sing(R)} T(R_\m)
\]
of \cref{t:transhom} is given by $\psi([M]) = (\rank_{\fp_1}(M), \ldots, \rank_{\fp_k}(M), ([M_\m])_{\m \in \Sing(R)})$ and has image defined by $\rank_{\fp_i}(M) = \rank_{\fp_i}(M_\m)$ for all $\m \in \Sing(R)$ and $\fp_i \subseteq \m$. Since each $T(R_{\m})$ is free (\cref{t:local-factorial}), the codomain of $\psi$ is free. We now describe its image as a Diophantine monoid having defining matrix with columns indexed by $\minspec{R}$ and by the isomorphism classes of indecomposable lattices over the localizations of $R$.

Let $\Sing(R)=\{\m_1, \ldots, \m_{b_1}, \ldots, \m_b\}$ with $b=b_1+b_2$ so that $\m_1$, $\ldots\,$,~$\m_{b_1}$ are the maximal ideals in $\Sing(R)$ containing exactly one minimal prime ideal and $\m_{b_1+1}$,~$\ldots\,$,~$\m_b$ are the maximal ideals  in $\Sing(R)$ containing exactly two minimal prime ideals. For each $i\in [1,b]$, fix $A_{i1}, \ldots, A_{it_i}$ representatives of the finitely many indecomposable $R_{\m_i}$-lattices and for an $R$-lattice $M$, write $M_{\m_i}\cong A_{i1}^{e_{i1}}\oplus \cdots \oplus A_{it_i}^{e_{it_i}}.$
If $R_{\m_i}$ is a domain containing only the minimal prime $\fp$, then $\sum_{j=1}^{t_i}e_{ij}=\rank_\fp(M)$. If $R_{\m_i}$ has two minimal primes $\fp R_{\m_i}$ and $\fq R_{\m_i}$, then assume $\rank(A_{i1})=(1,0)$, $\rank(A_{i2})=(0,1)$, and $\rank(A_{ij})=(1,1)$ for all $j\in [3,t_i]$. Then $\rank_\fp(M)=\left(\sum_{j=1}^{t_i} e_{ij}\right)-e_{i2}$ and $\rank_\fq(M)=\left(\sum_{j=1}^{t_i} e_{ij}\right)-e_{i1}$. 

We now define a $(b_1+2b_2)\times (k+\sum_{i=1}^bt_i)$ matrix  $B$. The first $k$ columns record which minimal primes of $R$ are contained in which maximal ideals. If $1\leq i\leq b_1$ and $1\leq j\leq k$, the $ij$th entry of $B$ is $1$ if $\fp_i\in \m_j$ and $0$ otherwise. If $b_1<i\leq b$, there are two rows of $B$ associated to $\m_i$; $\fp_u$ and $\fp_v$ are distinct minimal primes contained in $\m_i$ and so we place a $1$ in the $b_1+2(i-b_1)$st row and $u$th column and in the $b_1+2(i-b_1)-1$st row and $v$th column. The other columns record the ranks of the indecomposable modules over the various localizations of $R$. If $A_{ij}$ is an indecomposable $\m_i$-lattice with $1\leq i\leq b_1$, then there is a $-1$ in the $i$th row and $A_{ij}$th column. If $b_1< i\leq b$, then the $b_1+2(i-b_1)-1$st row, restricted to columns indexed by $A_{i1}, \ldots, A_{it_i}$ is $\begin{bmatrix} -1 & 0 & -1 & \cdots & -1\end{bmatrix}$ and the $b_1+2(i-b_1)$st row, restricted to columns indexed by $A_{i1}, \ldots, A_{it_i}$ is $\begin{bmatrix} 0 & -1 & -1 & \cdots & -1\end{bmatrix}$.

Observe that the entries of the first $k$ columns of $B$ are all in $\{0,1\}$ while the entries of the remaining $n - k$ columns are all in $\{0,-1\}$. With this setup, and with the description of $\psi(T(R))$ given in \cref{t:transhom}, we have the following result.

\begin{proposition} \label{p:transfer-diophantine}
Let $R$ be a Bass ring and let $B$ be as defined above. The map $\psi$ given in \cref{t:transhom} gives a transfer homomorphism from $T(R)$ to $\ker(B)\cap \bN_0^n$. In particular, $T(R)$ is transfer Krull.
\end{proposition}

\begin{remark}\label{rem:noncancellative}
    Although $\psi(T(R))$ is always Krull when $R$ is a Bass ring, the monoid $T(R)$ itself is rarely Krull since it is not, in general, cancellative. In fact, by \cite[Theorem 2.7]{wiegand-84}, the monoid $T(R)$ is cancellative if and only if the natural map $\pic(R)\rightarrow \pic(\overline R)$ is injective, and this rarely happens. However, as is observed in \cite[Theorem 6.2]{levy-wiegand85}, an isomorphism $M\oplus X\cong N\oplus X$ implies $M\cong N$ whenever $M$, $N$ and $X$ are lattices with $X$ projective. In fact, the monoid $T(R)$ is unit-cancellative since we consider only finitely generated modules over a commutative ring.
\end{remark}

\begin{example} \label{exm:bass-domain}
If $R$ is a domain, by \cite{levy-odenthal96a}, we already know that $T(R)$ is half-factorial. However, we explicitly build $B$ to illustrate the construction. Every maximal ideal of $R$ contains the unique minimal ideal of $R$ and thus the only positive entries of $B$ are in the first column, consisting of all $1$s. Moreover, for each $\m\in \maxspec(R)$, we have $\rank_{R_\m}(A)=1$ for each each indecomposable $R_\m$-lattice $A$. Thus \[B=\begin{bmatrix} 1 & -1 & \cdots & -1 & 0 & \cdots & 0 & \cdots & 0 & \cdots & 0 \\ 1 & 0 & \cdots & 0 & -1 & \cdots & -1 & \cdots & 0 & \cdots & 0 \\ &&\vdots&  &&&& \\ 1 & 0 & \cdots & 0 & 0 & \cdots & 0 & \cdots & -1 & \cdots & -1 \\ \end{bmatrix}.\] From here it is obvious that $\psi(T(R))\cong \ker(B)\cap \mathbb N_0^n$ is factorial.
\end{example}

\begin{example} \label{exm:bass-ngon}
Fix $m\geq 3$ and for each $i\in [1,m]$, set $c_i=\cos(2\pi i/m)$, $s_i=\sin(2\pi i/m)$, and $t_i=\frac{s_{i+1}-s_i}{c_{i+1}-c_i}$. Then the lines $y-s_i=t_i(x-c_i)$ give $m$ lines in the real plane containing the $m$ edges of a regular $m$-gon. Set \[ R=\mathbb R[x,y]/(\{\,y-s_i=t_i(x-c_i) : i\in [1,m]\,\}. \] With $\m_i=(x-c_i, y-s_i)$ for each $i\in [1,m]$, set $U=R\backslash \bigcup_{i=1}^m\m_i$. Then $R_U$ is a semilocal Bass ring with $m$ minimal prime ideals corresponding to the $m$ edges of the $m$-gon and $m$ maximal ideals corresponding to the $m$ vertices. Moreover, each prime ideal is contained in exactly two maximal ideals and each maximal ideal contains exactly two minimal prime ideals. Consequently (after removing duplicate columns), \[B=\begin{bmatrix}1 & 0 & 0 & 0 & \cdots & 0 & -1 & 0 & -1 & 0 & 0 & 0 & \cdots & 0 & 0 & 0 \\ 0 & 1 & 0 & 0 & \cdots & 0 & 0 & -1 & -1 & 0 & 0 & 0 & \cdots & 0 & 0 & 0 \\ 0 & 1 & 0 & 0 & \cdots & 0 & 0 & 0 & 0 & -1 & 0 & -1 & \cdots & 0 & 0 & 0 \\ 0 & 0 & 1 & 0 & \cdots & 0 & 0 & 0 & 0 & 0 & -1 & -1 & \cdots & 0 & 0 & 0 \\ & & & & \vdots & & & & & & & & \vdots & & & \\ 0 & 0 & 0 & 0 & \cdots & 1 & 0 & 0 & 0 & 0 & 0 & 0 & \cdots & -1 & 0 & -1  \\ 1 & 0 & 0 & 0 & \cdots & 0 & 0 & 0 & 0 & 0 & 0 & 0 & \cdots & 0 & -1 & -1 \end{bmatrix}.\] 
\end{example}

\begin{remark}\leavevmode
\begin{enumerate}
\item By construction, columns corresponding to distinct maximal ideals of $R$ cannot be the same. Also, no two of the first $k$ columns can be the same. However, for each maximal ideal $\m$ of $R$ there can be some duplicate columns, corresponding to multiple nonisomorphic indecomposables with the same rank; $1$ in the domain case and $(1,1)$ in the non-domain case. \cref{l:duplicate columns} allows us to consider the Diophantine monoid with these duplicate columns removed.

\item Set $H=\ker(B)\cap \mathbb N_0^n$. If $\mathbf e_i$ is the $i$-th standard unit vector, then $-\mathbf e_i$ appears as a column. Consequently, the inclusion $H\subseteq \mathbb N_0^n$ is a cofinal divisor homomorphism with class group $\im(B)\cong \mathbb Z^{b_1+2b_2}$ and with set of prime divisors given by the distinct columns of $B$. 
Thus, the monoid $H$ satisfies \cite[Proposition 6.2]{baeth-geroldinger14}, giving finiteness results for several arithmetical invariants. 

While the inclusion $H\subseteq \mathbb N_0^n$ is a cofinal divisor homomorphism, it is not a divisor theory.
A divisor theory can be obtained by eliminating the first $k$ variables, corresponding to the minimal prime ideals.
This is carried out in detail in \cref{t:divtheory} below for monoids of graph agglomerations.
These monoids of graph agglomerations are in fact precisely monoids of the form $\ker(C) \cap \bN_0^{n}$ with $C$ arising from $B$ as above by de-duplicating columns.
\end{enumerate}
\end{remark}

The transfer homomorphism from \cref{p:transfer-diophantine} can be simplified using \cref{l:duplicate columns}, giving rise to the following result that we will prove in \cref{s:tie-in}.

\begin{proposition} \label{p:transfer-deduplicated}
   Let $R$ be a Bass ring.
   Let $E \subseteq \maxspec(R)$ denote the set of maximal ideals of $R$ that contain two minimal prime ideals, and let $V\coloneqq \minspec(R)$.
   Let $c\coloneqq \card{V}+3\card{E}$, and let the coordinates of $\bN_0^c$ be indexed as follows: the first $\card{V}$ coordinates are indexed by elements of $V$.
   For each $\fm \in E$ there are three further coordinates, indexed by $\fm_{\fp}$, by $\fm_{\fq}$, and by $\fm$, where $\fp$, $\fq$ are the minimal prime ideals contained in $\fm$. 
   Let 
   \[
    H = \{\, \vec x \in \bN_0^c : x_\fp = x_\fm + x_{\fm_{\fp}} \text{ for all $(\fm,\fp) \in E \times V$ with $\fp \subseteq \fm$ } \,\}.
   \]
   Then there exists a transfer homomorphism $\varphi\colon T(R) \to H$.
\end{proposition}

\begin{proof}
   Let $C$ be the $2\card{E} \times c$ integer matrix, defined as follows:
   the rows are indexed by by pairs $(\fm,\fp) \in E \times V$ with $\fp \subseteq \fm$.
   The columns are indexed analogously to the coordinates of $\bN_0^c$.
   Denoting by $\vec e_i$ the $\card{V}+3\card{E}$-dimensional standard row vector with $1$ in coordinate $i$ and $0$ everywhere else, the row $(\fm,\fp)$ of $C$ is defined to be equal to $\vec e_{\fp} -  \vec e_{\fm} - \vec e_{\fm_{\fp}}$.
   Thus $H = \ker(C) \cap \bN_0^c$.
   
   We now note that $C$ arises from the matrix $B$ of \cref{p:transfer-diophantine} as follows:
   \begin{itemize}
       \item for each maximal ideal containing a unique minimal prime ideal, the corresponding columns are erased.
       \item If $\fm_i$ is one of the maximal ideals containing two minimal prime ideals ($b_1 < i \le b$), the columns corresponding to the rank $(1,1)$ indecomposables $A_{i,3}$, $\ldots\,$,~$A_{i,t_i}$, which are duplicates of each other, are replaced by a single copy, which becomes the column indexed by $\fm$ in $C$.
       The columns corresponding to the rank $(1,0)$ and $(0,1)$ indecomposables remain unchanged; they are the columns indexed by $\fm_\fp$ and $\fm_\fq$.
   \end{itemize}
   The first of these modifications yields an isomorphic Diophantine monoid, because each removed column has a single nonzero entry, which is $-1$, and therefore the removal is simply eliminating a redundant variable.
   The second modification preserves the transfer homomorphism by \cref{l:duplicate columns}.
\end{proof}
 
As is clear from the description of $\psi(T(R))$ as a Diophantine monoid, the arithmetic of $T(R)$ is governed by the structure of $\Spec(R)$. We will make this more precise in \cref{s:tie-in} after a new object is introduced in \cref{s:agglomerations}, but we already state the main results as they apply to $T(R)$. 
The description uses the prime ideal intersection graph (\cref{d:prime-ideal-intersection-graph}) alluded to in the introduction.

\begin{theorem} \label{t:main-finiteness}
Let $R$ be a nonzero Bass ring and let $\cG_R$ be the graph of prime ideal intersections of $R$.
\begin{enumerate}
\item\label{mf:elasticities-general} 
The elasticity $\rho(T(R)) \in \bQ$ is finite and accepted, and
\[
\{\, \rho([M]) : [M] \in T(R) \,\} = \{\, q \in \bQ : 1 \le q \le \rho(T(R)) \,\}.
\]
\item\label{mf:elasticities} If $m$ is the maximal order of a connected component in $\cG_R$ and $D$ is the maximal degree of a vertex in $\cG_R$, then \[\rho(T(R))\leq m-\frac{m-1}{D}.\]
If $\cG$ is simple, then $\rho(T(R)) \le m - 2 + \tfrac{2}{m}$, with equality when $\cG$ is a complete graph.
\item\label{mf:refined-elasticities} For each $k$, we have $\rho_k(T(R)) \leq (k-1)|V|+1$.
Equality holds if and only if $\cG_R$ contains $k$ edge-disjoint spanning trees \textup{(}that is, the spanning tree packing number of the graph $\cG_R$ is at least $k$\textup{)}.
\item\label{mf:half-factorial} The monoid $T(R)$ is half-factorial if and only if $\cG_R$ is acyclic.
\item\label{mf:distances} The set of distances $\Delta(T(R))$ is finite.
\item\label{mf:uk} The monoid $T(R)$ satisfies the Strong Structure Theorem for Unions of Sets of Lengths.
In particular, there is $M \in \bN_0$ such that, for all $k \in \bN$, the unions $\mathcal U_k(T(R))$ are finite AAPs with difference $\min \Delta(T(R))$ and bound $M$.
\item\label{mf:lengths} The monoid $T(R)$ satisfies the Structure Theorem for Sets of Lengths: there is an $M \in \bN_0$ such that every $L \in \mathcal L(T(R))$ is an AAMP with difference $d \in \Delta(T(R))$ and bound $M$.
\item\label{mf:cat-omega} 
If $\Pic(T(R))$ is trivial \textup{(}for instance, when $R$ is semilocal\textup{)}, then $T(R)$ is a finitely generated Krull monoid.
We have $\omega(T(R)) < \infty$. The set of catenary degrees $\{\, \mathsf c([M]) : [M] \in T(R) \,\}$ is finite. In particular $\mathsf c(T(R)) < \infty$.
\item\label{mf:existence} For each $N\in \mathbb N$, there exists a Bass ring $R$ such that $\rho(T(R))\geq N$.
\end{enumerate}
\end{theorem}

\begin{proof}
See \cref{s:tie-in}.   
\end{proof}

\section{Monoids of graph agglomerations} \label{s:agglomerations}

In this section we depart from the theme of modules over Bass rings.
Instead we introduce the new, to our knowledge, class of monoids of graph agglomerations and study their factorization theory.
This will be justified by the final section of the paper, in which we identify the images of the transfer homomorphisms in \cref{s:modules} as monoids of graph agglomerations.
Thus, results on the factorization theory of monoids of graph agglomerations yield corresponding results on the factorization theory of monoids of modules over Bass rings.

As the monoids introduced here may be of independent interest to the factorization theory community, we have ensured that present section is readable independently from the rest of the paper (excluding, perhaps, the occasional glance at the background material presented in \cref{s:background}).
Keep in mind that our graphs are finite and permit multiple edges but no loops.

\begin{definition} \label{def:agglomeration}
    Let $\cG=(V,E,r)$ be a graph.
    \begin{enumerate}
        \item An \defit{agglomeration on $\cG$} is a function $a \colon V \cup E \to \bN_0$ such that $a(v) \ge a(e)$ whenever $v \in V$ and $e \in E$ are incident.
    \item  The \defit{monoid of (graph) agglomerations on $\cG$}, denoted by $\agg(\cG)$ is the set of all agglomerations on the graph $\cG$ together with pointwise addition.
    \end{enumerate}
\end{definition}

Observe that a sum of two agglomerations is an agglomeration, and that the function that is identically zero is an agglomeration, so that $\agg(\cG)$ is indeed a monoid.
As a submonoid of $\bN_0^{V \cup E}$, the monoid $\agg(\cG)$ is commutative, cancellative, and reduced.
If $\cG$ is the null graph, then $\agg(\cG)=0$ is the trivial group.
We tacitly ignore this trivial case when convenient.
If $\cG$ is a trivial graph, then $\agg(\cG)=(\bN_0,+)$ is a factorial monoid with the unique prime element $1$; we sometimes also exclude this trivial case for convenience.

The following two basic properties are immediate from the definition.

\begin{lemma} \label{l:substructure}
    Let $\cG$ be a graph.
    \begin{enumerate}
    \item If $\cG'$ is a subgraph of $\cG$, then $\agg(\cG')$ embeds canonically as a divisor-closed submonoid into $\agg(\cG)$.
    \item If $\cG_1 \oplus \cG_2$ is a disjoint union of two graphs, then
    \[
    \agg(\cG_1 \oplus \cG_2) \cong \agg(\cG_1) \times \agg(\cG_2).
    \]
    \end{enumerate}
\end{lemma}

Despite this lemma being trivial, it has two significant consequences for the study of factorizations in agglomeration monoids.
First, it generally suffices to consider connected graphs.
Second, it is possible to obtain lower bounds on arithmetical invariants by working with suitable subgraphs.

The definition of an agglomeration ensures that whenever $a(e) > 0$ for an edge $e$, then also $a(v)>0$ for the vertices $v$ incident with $e$; this allows us to define the support of an agglomeration.

\begin{definition} 
    Let $\cG=(V,E,r)$ be a graph.
    \begin{enumerate}
        \item If $a \in \agg(\cG)$ then the support of $a$, denoted by $\supp(a)$, is the subgraph of $\cG$ consisting of the vertices $v \in V$ with $a(v)>0$ and the edges $e \in E$ with $a(e) > 0$.
        \item If $\cG'=(V',E',r')$ is a subgraph of $\cG$, then the \defit{indicator} $\ind_{\cG'} \colon V \cup E \to \bN_0$ is the agglomeration defined by
        \[
        \ind_{\cG'}(x) =
        \begin{cases}
        1 & \text{if $x \in V' \cup E'$,} \\
        0 & \text{otherwise}.
        \end{cases}
        \]
    \end{enumerate}
\end{definition}

It is clear that the indicator and the support furnish a bijection between subgraphs of $\cG$ and agglomerations on $\cG$ taking values in $\{0,1\}$.
We shall tacitly identify subgraphs of $\cG$ with their indicators when convenient.
If $v \in V$, we write $\ind_{v}$ for the indicator which is $1$ on $v$ and $0$ everywhere else.
If $e \in E$, we write $\ind_{(r(e),e)}$ for the indicator which is $1$ on $e$ and its two incident vertices, but $0$ everywhere else. 

\begin{lemma}[Splitting Lemma] \label{splitting}
    Let $a$ be an agglomeration on a non-null graph $\cG=(V,E,r)$.
    \begin{enumerate}
    \item \label{splitting:max}
    Let $m = \max\{\, a(x) : x \in V \cup E \,\}$.
    Let $b \colon V \cup E \to \bN_0$ be defined by
    \[
    b(x) =
    \begin{cases}
    1 &\text{if $a(x)=m$,}\\
    0 &\text{if $a(x) < m$.}
    \end{cases}
    \]
    Then $b$ and $a-b$ are agglomerations, and therefore the agglomeration $a$ factors as $a = b + (a-b)$ in $\agg(\cG)$.
    
    \item \label{splitting:supp}
    The indicator $\ind_{\supp(a)}$ and $a - \ind_{\supp(a)}$ are agglomerations, and therefore the agglomeration $a$ factors as $a = \ind_{\supp(a)} + (a-\ind_{\supp(a)})$ in $\agg(\cG)$.
    \end{enumerate}
\end{lemma}

\begin{proof}
\ref{splitting:max}
    If $e$ is an edge with $a(e)=m$, then $b(v) \ge m$ for any incident vertex $v$.
    Since $m$ is the maximum of the function $a$, we must have $b(v)=m$.
    Thus $b$ is an agglomeration.
    
    Consider now $a-b$.
    We only have to verify the inequality for vertices $v$ with $a(v)=m$.
    Then $(a-b)(v)=m-1$.
    Let $e$ be an edge incident with $v$.
    If $a(e) < m$, then $(a-b)(e)=a(e) \le m-1 = (a-b)(v)$.
    Otherwise, we have $a(e)=m$, but then $(a-b)(e) = a(e)-1 = m-1 = (a-b)(v)$.
    
\ref{splitting:supp}
    It is easy to see that in $a - \ind_{\supp(a)}$ all positive values are decreased by $1$, and therefore the inequalities for an agglomeration are satisfied. 
\end{proof}

Observe that in \ref{splitting:max} of the previous lemma, the agglomeration $b$ is simply the indicator of the subgraph of $\cG$ at which $a$ is maximal.

\begin{proposition} \label{p:agg-atoms}
    Let $\cG$ be a graph.
    An agglomeration $a \in \agg(\cG)$ is an atom in $\agg(\cG)$ if and only if $a=\ind_{\cG'}$ for a non-null, connected subgraph of $\cG'$.
\end{proposition}

\begin{proof}
    If $a$ is an atom, the splitting lemma implies $a(x) \in \{0,1\}$ for all $x \in V \cup E$.
    Thus necessarily $a = \ind_{\cG'}$ for a subgraph $\cG'$ of $\cG$.
    Now $\ind_{\cG'} = \ind_{\cG_1'} + \ind_{\cG_2'}$ if and only if $\cG'$ is the disjoint union of $\cG_1'$ and $\cG_2'$.
    The claim follows.
\end{proof}

The splitting lemma implies that every nonzero $a \in \agg(\cG)$ is a sum of atoms, that is, the monoid $\agg(\cG)$ is atomic (this is also obvious from more general facts).
Having in mind the bijection between subgraphs and indicator functions, the monoid $\agg(\cG)$ is generated by the connected subgraphs of $\cG$.
Hence, an arbitrary element of $\agg(\cG)$ may be viewed as an \emph{agglomeration} of subgraphs of $\cG$, explaining our terminology.

The next two results place $\agg(\cG)$ in the much larger class of Krull monoids.
For Krull monoids there is an well-established machinery to study their factorization theory via transfer homomorphisms to monoids of zero-sum sequences.
Surprisingly, this does \emph{not} seem to be the best approach to study agglomeration monoids on graphs; after a short interlude we will therefore return to study the factorization theory of agglomeration monoids directly.

\begin{theorem} \label{t:divtheory}
    Let $\cG=(V,E,r)$ be a graph.
    The monoid of agglomerations $\agg(\cG)$ is a finitely generated reduced Krull monoid and its divisor theory has $3\card{E} + \iota$ prime divisors.
    Here, $\iota$ is the number of isolated vertices in $\cG$.
\end{theorem}

\begin{proof}
    It suffices to consider the case where $\cG$ is connected.
    If $\cG$ has no edges, then it is either the null graph, in which case $\agg(\cG)=0$, or it consists of a single vertex, in which case $\agg(G)=\bN_0$, which is a Krull monoid with a single prime divisor.
    From now on we may assume that $\cG$ has at least one edge.
    
    By definition $\agg(\cG)$ is the submonoid of $\bN^{V \cup E}$ defined by the inequalities $a(e) \le a(v)$ whenever $e \in E$ is incident with $v \in V$.
    We shall explicitly determine a divisor theory for $\agg(\cG)$; we do so through the standard technique of introducing slack variables to turn the defining inequalities into equations and then eliminate redundant variables.
    
    Let $I = \{\, (e,v) \in E \times V : e \sim v \,\}$. Note $\card{I} = 2\card{E}$; we may assume without restriction $E \cap (E \times V) = \emptyset$.
    For $a \in \agg(\cG)$ we define $f_a \colon E \cup I \to \bN_0$ by 
    \[
    f_a(e)=a(e) \text{ for $e \in E$} \quad\text{and} \quad f_a(e,v) = a(v) - a(e) \text{ for $(e,v) \in I$}.
    \] 
    Let $\varphi\colon \agg(\cG) \to \bN_0^{E \cup I}$ be defined by $\varphi(a)=f_a$.
    Then $\varphi$ is a monoid homomorphism.
    We show that $\varphi$ is injective.
    Let $a$,~$b \in \agg(\cG)$ with $f_a=f_b$.
    Then $a(e)=b(e)$ for all $e \in E$.
    Let $v \in V$ and let $e$ be an edge incident with $v$.
    Then $a(v) = f_a(e,v) + f_a(e) = f_b(e,v) + f_b(v) = b(v)$.
    
    It is similarly easy to see that $\varphi(\agg(\cG))$ is the submonoid of all $f \in \bN_0^{E \cup I}$ satisfying the equations $f(e,v) + f(e) = f(e',v) + f(e')$ whenever $e$,~$e' \in E$ are both incident with $v$.
    As a Diophantine monoid, therefore $\varphi(\agg(\cG)) \subseteq \bN_0^{E \cup I}$ is saturated.
    Hence $\varphi \colon \agg(\cG) \to \bN_0^{E \cup I}$ is a divisor homomorphism.
    
    To show that $\varphi$ is a divisor theory, it suffices to show that for every standard basis vector $f$ of $\bN_0^{E \cup I}$ there exist $a$,~$b \in \agg(\cG)$ with $f = \min\{ \varphi(a), \varphi(b) \}$, where the minimum is taken pointwise.
    
    We distinguish two cases for $f \in \bN_0^{E \cup I}$. 
    Let first $e \in E$ and let $f(e)=1$ and $f(x)=0$ for all $x \in E\setminus\{e\} \cup I$.
    Let $v$,~$w$ be the two vertices incident with $e$.
    Define $a \in \agg(\cG)$ by $a(e)=a(v)=a(w)=1$ and $a(x)=0$ if $x \not \in \{v,w,e\}$.
    Then $f_a(e)=1$ and $f_a(e')=0$ for $e' \in E \setminus \{e\}$. 
    Now define $b \in \agg(\cG)$ by $b(v')=b(e')=1$ for all $v' \in V$ and $e' \in E$.
    Then $f_b(e')=1$ for all $e' \in E$ and $f_b(e',v')=0$ for all $(e',v')\in I$.
    We conclude $f = \min\{f_a,f_b\}$.
    
    Let now $(e,v) \in I$ and let $f(e,v)=1$ and $f(x)=0$ for all $x \in E \cup I \setminus (e,v)$.
    Define $a \in \agg(\cG)$ by $a(v) = 1$ and $a(x)=0$ for all $x \ne v$.
    Then $f_a(e)=0$ for all $e \in E$.
    For the incidences, we have $f_a(e',v)=1$ if $e' \sim v$ and $f_a(e',v')=0$ for all $(e',v') \in I$ with $v' \ne v$.
    Now define $b \in \agg(\cG)$ by $b(v')=1$ for all $v' \in V$, by $b(e)=0$, and $b(e') = 1$ if $e' \in E \setminus \{e\}$.
    Then $f_b(e,v') = 1$ for both vertices $v'$ with $v' \sim e$.
    Moreover $f_b(e',v')=0$ for all $(e',v') \in I$ with $e' \ne e$.
    We again conclude $f = \min\{f_a,f_b\}$.
\end{proof}

\begin{corollary} \label{c:class-group}
    The divisor class group of $\agg(\cG)$ is free abelian of rank $2\card{E} - \card{V} + \iota$, where $\iota$ is the number of isolated vertices in $\cG$.
\end{corollary}

\begin{proof}
    The divisor class group of a finite product of Krull monoid is isomorphic to the product of the individual class groups.
    Without restriction, let $\cG$ be connected and non-trivial.
    
    We continue in the notation of the proof of \cref{t:divtheory}.
    The image of the divisor theory $\varphi\colon \agg(\cG) \to \bN_0^{E \cup I}$ is defined by the equations $f(e,v)+f(e)=f(e',v)+f(e')$ whenever $e$,~$e'$ are edges incident with a vertex $v$.
    Fixing for each vertex $v$ an edge $e_v$ incident with $v$, we observe that the equations
    \[
    f(e_v,v)+f(e_v) = f(e',v) + f(e'), \quad\text{whenever $e' \in E\setminus\{e_v\}$ is incident with $v$,}
    \]
    are sufficient to define $\im \varphi$.
    This yields $\card{I} - \card{V} + \iota = 2\card{E} - \card{V} + \iota$ linear equations defining $\im \varphi$ as a submonoid of $\bN_0^{E\cup I}$.
    These equations are all linearly independent because $f(e',v)$ appears in a unique equation.
    We may write the corresponding system matrix as an integer matrix $M$ with columns indexed by $E \cup I$, and having $2\card{E}- \card{V} + \iota$ linearly independent rows.
    
    Denote by $\quo(\im\varphi)$ the quotient group of $\im\varphi$.
    Then the divisor class group $\Cl(\agg(\cG))$ is isomorphic to $\bZ^{E \cup I} / \quo(\im \varphi)$.
    We claim $\quo(\im\varphi) = \ker M$.
    Indeed the inclusion $\quo(\im\varphi) \subseteq \ker M$ is obvious.
    For the converse, note that the agglomeration $b \in \agg(\cG)$ with $b(v)=2$ for all $v \in V$ and $b(e)=1$ for all $e \in E$ yields a vector $f_b \in \ker M$ all of whose entries are positive.
    Now let $x \in \ker(M)$.
    Let $n \in \bN_0$ be sufficiently large so that $x+n f_b \in \bN_0^{E \cup I} \cap \ker(M)$.
    Then $nf_b$,~$x+nf_b \in \im\varphi$ and $x = (x+n f_b) - n f_b \in \quo(\im \varphi)$.
    
    We conclude $\Cl(\agg(\cG)) \cong \bZ^{E \cup I}/\ker(M) \cong \im M$.
    Hence $\Cl(\agg(\cG))$ is free abelian of rank $2\card{E} - \card{V} + \iota$.
\end{proof}

\begin{example}
    Let $\cG$ be the cycle graph of length $3$ with vertices $v_1$, $v_2$, $v_3$ and edges $e_{12}$, $e_{23}$, $e_{13}$.
    The divisor theory maps $\agg(\cG)$ into $\bN_0^9$, explicitly
    \[
    \begin{split}
    \varphi(a) = \big(&a(e_{23}), a(e_{13}), a(e_{12}), a(v_1) - a(e_{12}), a(v_1) - a(e_{13}),\\
    &a(v_2) - a(e_{12}), a(v_2) - a(e_{23}), a(v_3) - a(e_{13}), a(v_3) - a(e_{23})\big).
    \end{split}
    \]
    Then $\im\varphi = \bN_0^9 \cap \ker(M)$ with
    \[
    M =
    \begin{bmatrix}
    0 & 1 & -1 & -1 & 1 & 0 & 0 & 0 & 0 \\
    1 & 0 & -1 & 0 & 0 & -1 & 1 & 0 & 0 \\
    1 & -1 & 0 & 0 & 0 & 0 & 0 & -1 & 1
    \end{bmatrix}.
    \]
    The divisor class group $G \coloneqq \Cl(\agg(\cG)) \cong \bZ^9/\ker(A) \cong \im(A)$ is isomorphic to $\bZ^3$.
    In the study of Krull monoids, it is also important to understand the subset $G_0$ of $G$ containing prime divisors, that is, the image of the $9$ standard basis vectors under the map $\bZ^9 \to \bZ^9/\ker(A)$.
    This is just the set of columns of $A$.
    In this example $G_0 \ne -G_0$, which is atypical for the Krull monoids that are usually being studied, but typical for agglomeration monoids.
    As a consequence, many known results about Krull monoids are not applicable in our setting.
\end{example}

Since $\agg(\cG)$ is a finitely generated Krull monoid, all the main finiteness results (\cref{t:krull-finiteness}) about such monoids hold.
We continue with a more detailed investigation of the arithmetic of $\agg(\cG)$.

\begin{definition}
    Let $\cG=(E,V,r)$ be a non-null graph and $a \in \agg(\cG)$.
    \begin{enumerate}
        \item The \defit{sequence-length} of $a$ is 
        \[  
        \ell(a) \coloneqq \sum_{v \in V} \deg_{\cG}(v) a(v) - \sum_{e\in E} a(e) \in \bN_0.
        \]
        \item The \defit{Davenport constant} of $\agg(\cG)$ is
        \[
        \sD(\agg(\cG)) \coloneqq \max\{\, \ell(a) : a \text{ an atom of $\agg(\cG)$} \,\}.
        \]
    \end{enumerate}
\end{definition}

Using the defining property of an agglomeration, it is easy to verify that $\ell(a)$ is indeed nonnegative.

\begin{remark}
A divisor theory of a Krull monoid gives rise to a transfer homomorphism to a monoid of zero-sum sequences $\mathcal B(G_0)$, where $G_0$ is a subset of  the divisor class group $G$ (see the surveys \cite{baeth-wiegand13,geroldinger16,geroldinger-zhong20}).
In this way, to every $a \in \agg(\cG)$, we can associate a zero-sum sequence $\theta(a) \in \mathcal B(G_0)$.
The sequence-length $\ell$ is defined in such a way that $\ell(a)$ is indeed the length of the zero-sum sequence $\theta(a)$.
This can be seen from the divisor theory as constructed in the proof of \cref{t:divtheory}.
By the uniqueness properties of a divisor theory, this length does not depend on the particular choice of divisor theory.

The Davenport constant is a classical invariant in factorization theory through which many arithmetical invariants can be estimated, or even expressed exactly.
In many cases, e.g., when $G=G_0$ is a finite abelian group, it is notoriously difficult to compute the Davenport constant.
The case of agglomeration monoids is remarkable in that it is very easy to compute $\sD(\agg(\cG))$.
However, due to the set of prime divisors not being symmetric in our setting ($G_0 \ne -G_0$) this is of limited use. 
For instance, $\sD(\agg(\cG))$ gives bounds on invariants such as the refined elasticities $\rho_k(\agg(\cG))$, but in the case $G_0 \ne -G_0$ it does not give precise values.
\end{remark}

\begin{proposition} \label{p:davenport}
    Let $\cG = (E,V,r)$ be a non-trivial, connected graph.
    \begin{enumerate}
        \item\label{davenport:formula} We have
        \[
        \sD(\agg(\cG)) = 2\card{E} - \card{V} + 1.
        \]
        \item\label{davenport:extremal} 
        An atom $a$ of $\agg(\cG)$ satisfies $\ell(a)=\sD(\agg(\cG))$ if and only if $\supp(a)$ is a tree containing every vertex of degree at least $2$.
        In particular, every spanning tree of $\cG$ yields an atom of maximal sequence-length.
        If $\cG$ has minimum vertex degree at least $2$, then the atoms of maximal sequence-length are in bijection with the spanning trees of $\cG$.
    \end{enumerate}
\end{proposition}

\begin{proof}
    \ref{davenport:formula}
    Let $a \in \agg(\cG)$ be an atom.
    Then $a$ is the indicator of the connected subgraph $\supp(a)=(V_a,E_a,r_a)$ of $\cG$. 
    Therefore we get
    \begin{equation} \label{eq:atomlength}
    \ell(a) = \sum_{v \in V_a} \deg_{\cG}(v) - \card{E_a}.
    \end{equation}
    Note that the degree of $v$ is considered as a vertex of $\cG$, \emph{not} as a vertex of $\supp(a)$.
    
    If there is a vertex $v \in V\setminus V_a$, then $v$ can be chosen in such a way that it can be connected to $\supp(a)$ by a single edge.
    Adding $v$, and a connecting edge, to $\supp(a)$ produces an atom $a'$ with $\ell(a') \ge \ell(a)$.
    To maximize \cref{eq:atomlength}, we may therefore assume $V_a=V$.
    Then
    \[
        \ell(a) = \sum_{v \in V} \deg(v) - \card{E_a} = 2 \card{E} - \card{E_a}.
    \]
    This value is maximized by taking $\card{E_a}$ minimal, which happens if and only if $\supp(a)$ is a spanning tree.
    In this case $\card{E_a} = \card{V} - 1$.
    
    \ref{davenport:extremal}
    Let again $a$ be an atom with $\supp(a)=(V_a,E_a,r_a)$.
    If there exists $v \in V \setminus V_a$ and $\deg(v) > 1$, then the extension process just described produces an atom $a'$ with $\ell(a') > \ell(a)$ and $\ell(a)$ cannot be maximal.
    Similarly, if $\supp(a)$ is not a tree we may remove an edge, without breaking the connectivity, to find an atom $a'$ with $\ell(a') > \ell(a)$.
    Thus, if $\ell(a) = \sD(\agg(\cG))$, then $\supp(a)$ is a tree and contains every vertex of $\cG$ of degree at least $2$.
    
    Conversely, if $\supp(a)$ is of this form, we may extend $a$ to $a'$ with $\supp(a')$ a spanning tree and $\ell(a)=\ell(a')$.
    By the argument in \ref{davenport:formula}, this sequence-length is maximal.
\end{proof}

In the following lemma we characterize prime and absolutely irreducible elements of $\agg(\cG)$ and show that the indicator of a graph decomposes uniquely as the sum of the indicators of its connected components.

\begin{lemma} \label{l:nicefact}
    Let $\cG=(V,E,r)$ be a non-null graph.
    \begin{enumerate}
        \item\label{nicefact:primes} An atom $a\in \agg(\cG)$ with $\supp(a)=(V', E', r')$ is prime if and only if $\deg_{\cG}(v)\leq 1$ for all $v\in V'$. 
        \item\label{nicefact:divisors} Let $a$ be an atom of $\agg(\cG)$ and $b \in \agg(\cG)$.
        If $a$ is a summand of $nb$ for some $n \ge 1$, then $a$ is a summand of $b$.
        \item\label{nicefact:strongatom} Every atom of $\agg(\cG)$ is absolutely irreducible.
        \item\label{nicefact:disjoint} Let $\cG_1$, $\ldots\,$,~$\cG_s$ be pairwise disjoint, connected subgraphs of $\cG$.
        Then the factorization of $a=\ind_{\cG_1} + \cdots + \ind_{\cG_s}$ is unique \textup{(}up to order\textup{)}.
    \end{enumerate}
\end{lemma}

\begin{proof}
    \ref{nicefact:primes}
    We may without restriction assume that $\cG$ is connected.
    Suppose that $a$ is an atom of $\agg(\cG)$ with $\supp(a)=(V', E', r')$ and $\deg_{\cG}(v)\leq 1$ for all $v\in V'$. Then $\deg_{\supp(a)}(v)=1$ for each $v\in V'$ as well and, since $\supp(a)$ is connected, either $\supp(a)=\ind_{v}$ is a trivial graph or $\supp(a)=\ind_{(r(e),e)}$ for some edge $e$. In the latter case, we must have $\cG = \ind_{(r(e),e)}$ as well. With $v_1$ and $v_2$ the only vertices of $\cG$ and $e$ its only edge, the only atoms in $\agg(\cG)$ are $\ind_{v_1}$, $\ind_{v_2}$, and $a=\ind_{\cG}$. Clearly $a$ is prime. In the former case, we have $V'=\{v\}$ and either $V=\{v\}$ or there is a unique $e\in E$ with $v\in r(e)$. If $V=\{v\}$, then clearly $a$ is prime. Otherwise, suppose $a+b=c+d$ for some $b$, $c$,~$d \in \agg(\cG)$.
    Then $a(v)=1$ and $a(e)=0$ implies $c(v)+d(v)=a(v)+b(v) > a(e)+b(e) = c(e)+d(e)$. Without restriction $c(v) > c(e)$, and therefore $a$ is a summand of $c$ in $\agg(\cG)$.
    
    Now suppose $a\in \agg(\cG)$ is an atom with  $\supp(a)=(V', E', r')$ and such that there is $v\in V'$ with $\deg_{\cG}(v)\geq 2$.
    We have to show that $a$ is not prime.
    Consider $b \in \agg(\cG)$ with $b(v)=\deg_{\cG}(v)$ for every vertex $v$, and $b(e) = 1$ for every edge $e$.
    Observe that $a$ is a summand of $b$.
    Moreover,
    \[
    b = \sum_{e \in E} \ind_{(r(e),e)} = \ind_{\cG} + \sum_{v \in V} (\deg_{\cG}(v) - 1) \ind_{v}.
    \]
    Now, if $\card{E'}\ne 1$, then $a$ cannot be a summand of any $\ind_{(r(e),e)}$.
    If $\card{E'}=1$, then $a$ is neither a summand of $\ind_{\cG}$ nor of any $\ind_v$.
    In any case, the agglomeration $a$ is not a prime element.
    
    \ref{nicefact:divisors} Let $\supp(a)=(E_a,V_a,r_a)$.
    Note that $a$ is a summand of $nb$ if and only if, for every $v \in V_a$ and every $e \in E \setminus E_a$ with $e \sim v$, we have $nb(e) < nb(v)$.
    But then also $b(e) < b(v)$.
    Thus $a$ is a summand of $b$.
    
    \ref{nicefact:strongatom} By \ref{nicefact:divisors}.
    
    \ref{nicefact:disjoint}
    Clearly $\ind_{\cG_i}$ factors uniquely as an element of $\agg(\cG_i)$ since it is an atom there.
    The claim follows because $a$ is contained in the divisor-closed submonoid $\agg(\cG_1) \times \cdots \times \agg(\cG_s)$ of $\agg(\cG)$.
\end{proof}

\subsection{Elasticities}

We now turn our attention to specific bounds for the elasticity $\rho(\agg(\cG))$ and the refined elasticities $\rho_k(\agg(\cG))$.
After that, we characterize factorial and half-factorial monoids of graph agglomerations in \cref{t:factorial}.

A \defit{semi-length function} is a monoid homomorphism $\sigma\colon \agg(\cG) \to \mathbb R_{\ge 0}$ with $\sigma(a)=0$ if and only if $a=0$. 
By \cite{anderson-anderson92} we have $\rho(\agg(\cG))\leq M^*/m^*$
with
\begin{align*}
M^* &= \max\{\,\sigma(a) : \text{$a$ is an atom of $\agg(\cG)$}\,\} \quad\text{and} \\
m^* &= \min\{\,\sigma(a) : \text{$a$ is an atom of $\agg(\cG)$}\,\}.
\end{align*}

\begin{theorem} \label{thm:rho}
    Let $\cG$ be a non-null graph.
    Let $m$ be the maximal order of a connected component of $\cG$, and $D$ the maximal degree of $\cG$.

    \begin{enumerate}
        \item\label{rho:semilength} For every $r > D/2$, the function $\sigma_r\colon \agg(\cG) \to \bR_{\ge 0}$ defined by
         \[
            \sigma_r(a)=r \sum_{v \in V} a(v) - \sum_{e\in E} a(e)
        \]
         is a semi-length function.
        \item\label{rho:generic} We have $\rho(\agg(\cG)) \le m - \frac{m-1}{D}$.
        \item\label{rho:simple} If $\cG$ is simple, then $\rho(\agg(\cG)) \le m - 2 + \frac{2}{m}$.
    \end{enumerate}
\end{theorem}

\begin{proof}
    If $\cG_1$, $\ldots\,$,~$\cG_s$ are the connected components of $\cG$, it is easily shown that $\rho(\agg(\cG)) = \max\{\, \rho(\agg(\cG_i)) : i \in [1,s] \,\}$, see \cite[Proposition 1.4.5.2]{geroldinger-halterkoch06}.
    Since it also suffices to establish the semi-length property on each connected component, we assume without restriction that $\cG$ is connected and of order $m$.
    The claims are trivial for $m=1$, so assume $m \ge 2$.
    
    \ref{rho:semilength}
    Clearly $\sigma_r$ is additive on $\agg(\cG)$.
    Since  $r>\frac{D}{2}$ and $\cG$ has at most $Dm/2$ edges, we have $\sigma(r)\geq 0$ with equality only when $a=0$. 
    
    We now show \ref{rho:generic} and \ref{rho:simple} through a suitable choice of $r$.
    Suppose $r \ge 1$.
    From the definition of $\sigma_r$ and the fact that an atom has values in $\{0,1\}$, it is then clear that $\sigma_r(a)$ is maximized by taking $a$ an atom with $\supp(a)$ a spanning tree. Thus $M^*=rm - (m-1)$.

    \ref{rho:generic} Set $r = D \ge 1$. Let $a$ be an atom with $\supp(a)$ of order $m'$.
    If $m'=1$, then clearly $\sigma_r(a)=D$.
    If $m' > 1$, then $\sigma_r(a) \ge m'D - m'D/2 = m'D/2  \ge D$.
    Thus $m^*=D$ and $\rho(\agg(\cG)) \le m - (m-1)/D$.
    
    \ref{rho:simple}
    Set $r = m/2 \ge 1$ and let $a$ be an atom with $\supp(a)$ of order $m' \le m$.
    Since $\cG$ is simple, the support of $a$ has at most $m'(m'-1)/2$ edges.
    Thus
    \[
    \sigma_r(a) \ge r m' - \frac{m'(m'-1)}{2} = r + (m'-1)\Big( r - \frac{m'}{2} \Big) \ge r. 
    \]
    Thus $m^*=m/2$ and $\rho(\agg(\cG)) \le m - 2 + 2/m$.
\end{proof}

In some cases the previous result, with a suitable choice of $r$, allows us to determine the exact elasticity, as the next examples show.

\begin{example} \label{exm:elasticity} 
\mbox{}
\begin{enumerate}
    \item 
    Consider the $n$-cycle $\cG=\mathcal C_n$.
    For $i \in[1,n]$ let $\cT_i \subseteq \cC_n$ denote the subgraph in which we omit the $i$-th edge.
    Let $a = \cT_1 + \cdots + \cT_n$.
    Then $a(v)=n$ for all $v \in V$ and $a(e)=n-1$ for all $e \in E$.
    Thus $a=(n-1)\cC_n + \ind_{v_1} + \cdots + \ind_{v_n}$, where $v_i$ are the individual vertices. 
    (See \cref{fig:c4} for an illustration of this construction on the $4$-cycle.)
    We conclude $\rho(a) \ge 2 - \frac{1}{n}$.
    Since every proper subgraph of $\cC_n$ is a tree, one can check that with $r=n/(n-1)$ we have $\sigma_r(a) \ge r$ for all atoms $a$. 
    As in the previous theorem, the choice of $r$ gives a corresponding upper bound for the elasticity, and hence
    \[
        \rho(\agg(\cC_n)) = 2 - \frac{1}{n}.
    \]
    
    \item
    Suppose $\cG$ is a connected, simple $k$-regular graph on $n$ vertices.
    For each $v\in V$, define the subgraph $\cG_v=(V_v, E_v, r_v)$ with $E_v=\{\,e\in E : v\in r(e)\,\}$, with $V_v=\{\,r(e): e\in E_v\,\}$, and with $r_v=r|_{E_v}$.
    Then 
    \[
    \sum_{v\in V}\ind_{\cG_v}=\ind_{\cG}+\ind_{\cG}+(k-1)\sum_{v\in V}\ind_{v}
    \]
    are two factorizations of the agglomeration on $\cG$ assigning to every vertex the value $k+1$ and every edge the value $2$.
    (See \cref{fig:k4} for an illustration of this construction on the complete graph $\cK_4$.)
    Thus 
    \[
    \rho(\agg(\cG))\geq \frac{n(k-1)+2}{n}=(k-1)+\frac{2}{n}.
    \]
    
    \item
    Let $\cG=\mathcal K_n$ be a complete graph on $n$ vertices.
    The last example together with the upper bound from \ref{rho:simple} of \ref{thm:rho} shows
    \[
    \rho(\agg(\cG)) = n + \frac{2}{n} - 2.
    \]
    In particular, the upper bound for simple graphs in \cref{thm:rho} can be attained.
\end{enumerate}
\end{example}

\begin{figure}

\centering

\tikzstyle{vertex}=[circle, draw]

\noindent \begin{tikzpicture}[transform shape, scale=0.5]
\node[vertex](a) at (0, 5) {\LARGE $4$};
\node[vertex](b) at (0, 9) {\LARGE $4$};
\node[vertex](c) at (4, 5) {\LARGE $4$};
\node[vertex](d) at (4, 9) {\LARGE $4$};

\begin{scope}[every path/.style={-}, every node/.style={sloped, fill=white}]
\draw (a) -- node[midway, rotate=-90] {\LARGE $3$} (b);
\draw (b) -- node[midway,] {\LARGE $3$} (d);
\draw (d) -- node[midway, rotate=90] {\LARGE $3$} (c);
\draw (c) -- node[midway,] {\LARGE $3$} (a);
\end{scope} 

\node at (5,7) {\LARGE $=$};

\node[vertex](a) at (6, 5) {\LARGE $3$};
\node[vertex](b) at (6, 9) {\LARGE $3$};
\node[vertex](c) at (10, 5) {\LARGE $3$};
\node[vertex](d) at (10, 9) {\LARGE $3$};

\begin{scope}[every path/.style={-}, every node/.style={sloped, fill=white}]
\draw (a) -- node[midway, rotate=-90] {\LARGE $3$} (b);
\draw (b) -- node[midway] {\LARGE $3$} (d);
\draw (d) -- node[midway, rotate=90] {\LARGE $3$} (c);
\draw (c) -- node[midway] {\LARGE $3$} (a);
\end{scope}  

\node at (11, 7) {\LARGE $+$};

\node[vertex](a) at (12, 5) {\LARGE $1$};
\node[vertex](b) at (12, 9) {\LARGE $1$};
\node[vertex](c) at (16, 5) {\LARGE $1$};
\node[vertex](d) at (16, 9) {\LARGE $1$};

\node at (5,1.5) {\LARGE $=$};

\node[vertex](a) at (6, -0.5) {\LARGE $1$};
\node[vertex](b) at (6, 3.5) {\LARGE $1$};
\node[vertex](c) at (10, -0.5) {\LARGE $1$};
\node[vertex](d) at (10, 3.5) {\LARGE $1$};

\begin{scope}[every path/.style={-}, every node/.style={sloped, fill=white}]
\draw (a) -- node[midway, rotate=-90] {\LARGE $1$} (b);
\draw (b) -- node[midway] {\LARGE $1$} (d);
\draw (d) -- node[midway, rotate=90] {\LARGE $1$} (c);
\end{scope} 

\node at (11,1.5) {\LARGE $+$};

\node[vertex](a) at (12, -0.5) {\LARGE $1$};
\node[vertex](b) at (12, 3.5) {\LARGE $1$};
\node[vertex](c) at (16, -0.5) {\LARGE $1$};
\node[vertex](d) at (16, 3.5) {\LARGE $1$};

\begin{scope}[every path/.style={-}, every node/.style={sloped, fill=white}]
\draw (a) -- node[midway, rotate=-90] {\LARGE $1$} (b);
\draw (b) -- node[midway] {\LARGE $1$} (d);
\draw (a) -- node[midway] {\LARGE $1$} (c);
\end{scope} 

\node at (17,1.5) {\LARGE $+$};

\node[vertex](a) at (18, -0.5) {\LARGE $1$};
\node[vertex](b) at (18, 3.5) {\LARGE $1$};
\node[vertex](c) at (22, -0.5) {\LARGE $1$};
\node[vertex](d) at (22, 3.5) {\LARGE $1$};

\begin{scope}[every path/.style={-}, every node/.style={sloped, fill=white}]
\draw (a) -- node[midway, rotate=-90] {\LARGE $1$} (b);
\draw (a) -- node[midway] {\LARGE $1$} (c);
\draw (d) -- node[midway, rotate=90] {\LARGE $1$} (c);
\end{scope} 

\node at (23,1.5) {\LARGE $+$};

\node[vertex](a) at (24, -0.5) {\LARGE $1$};
\node[vertex](b) at (24, 3.5) {\LARGE $1$};
\node[vertex](c) at (28, -0.5) {\LARGE $1$};
\node[vertex](d) at (28, 3.5) {\LARGE $1$};

\begin{scope}[every path/.style={-}, every node/.style={sloped, fill=white}]
\draw (a) -- node[midway] {\LARGE $1$} (c);
\draw (b) -- node[midway] {\LARGE $1$} (d);
\draw (c) -- node[midway, rotate=-90] {\LARGE $1$} (d);
\end{scope}

\end{tikzpicture}

\caption{Decomposition in $\agg(\mathcal C_4)$ illustrating $\rho(\agg(\cC_n)) = 2 - \tfrac{1}{n}$. See \cref{exm:elasticity}.}    \label{fig:c4}
\end{figure}

\begin{figure}

\centering

\tikzstyle{vertex}=[circle, draw]

\noindent \begin{tikzpicture}[transform shape, scale=0.5]
\node[vertex](a) at (0, 8) {\LARGE $4$};
\node[vertex](b) at (6, 8) {\LARGE $4$};
\node[vertex](c) at (3, 10.5) {\LARGE $4$};
\node[vertex](d) at (3, 14) {\LARGE $4$};

\begin{scope}[every path/.style={-}]
\draw (a) -- node[midway, fill=white] {\LARGE $2$} (b);
\draw (a) -- node[midway, fill=white] {\LARGE $2$} (c);
\draw (a) -- node[midway, fill=white] {\LARGE $2$} (d);
\draw (b) -- node[midway, fill=white] {\LARGE $2$} (c);
\draw (b) -- node[midway, fill=white] {\LARGE $2$} (d);
\draw (c) -- node[midway, fill=white] {\LARGE $2$} (d);
\end{scope} 

\node at ((7,10.5) {\LARGE $=$};

\node[vertex](a) at (8, 8) {\LARGE $2$};
\node[vertex](b) at (14, 8) {\LARGE $2$};
\node[vertex](c) at (11, 10.5) {\LARGE $2$};
\node[vertex](d) at (11, 14) {\LARGE $2$};

\begin{scope}[every path/.style={-}]
\draw (a) -- node[midway, fill=white] {\LARGE $2$} (b);
\draw (a) -- node[midway, fill=white] {\LARGE $2$} (c);
\draw (a) -- node[midway, fill=white] {\LARGE $2$} (d);
\draw (b) -- node[midway, fill=white] {\LARGE $2$} (c);
\draw (b) -- node[midway, fill=white] {\LARGE $2$} (d);
\draw (c) -- node[midway, fill=white] {\LARGE $2$} (d);
\end{scope} 

\node at ((15,10.5) {\LARGE $+$};

\node[vertex](a) at (16, 8) {\LARGE $2$};
\node[vertex](b) at (22, 8) {\LARGE $2$};
\node[vertex](c) at (19, 10.5) {\LARGE $2$};
\node[vertex](d) at (19, 14) {\LARGE $2$};

\node at ((0,2.5) {\LARGE $=$};

\node[vertex](a) at (1, 0) {\LARGE $1$};
\node[vertex](b) at (7, 0) {\LARGE $1$};
\node[vertex](c) at (4, 2.5) {\LARGE $1$};
\node[vertex](d) at (4, 6) {\LARGE $1$};

\begin{scope}[every path/.style={-}]
\draw (a) -- node[midway, fill=white] {\LARGE $1$} (b);
\draw (a) -- node[midway, fill=white] {\LARGE $1$} (c);
\draw (a) -- node[midway, fill=white] {\LARGE $1$} (d);
\end{scope} 

\node at ((8,2.5) {\LARGE $+$};

\node[vertex](a) at (9, 0) {\LARGE $1$};
\node[vertex](b) at (15, 0) {\LARGE $1$};
\node[vertex](c) at (12, 2.5) {\LARGE $1$};
\node[vertex](d) at (12, 6) {\LARGE $1$};

\begin{scope}[every path/.style={-}]
\draw (a) -- node[midway, fill=white] {\LARGE $1$} (b);
\draw (b) -- node[midway, fill=white] {\LARGE $1$} (c);
\draw (b) -- node[midway, fill=white] {\LARGE $1$} (d);
\end{scope} 

\node at ((16,2.5) {\LARGE $+$};

\node[vertex](a) at (17, 0) {\LARGE $1$};
\node[vertex](b) at (23, 0) {\LARGE $1$};
\node[vertex](c) at (20, 2.5) {\LARGE $1$};
\node[vertex](d) at (20, 6) {\LARGE $1$};

\begin{scope}[every path/.style={-}]
\draw (a) -- node[midway, fill=white] {\LARGE $1$} (c);
\draw (b) -- node[midway, fill=white] {\LARGE $1$} (c);
\draw (c) -- node[midway, fill=white] {\LARGE $1$} (d);
\end{scope} 

\node at ((24,2.5) {\LARGE $+$};

\node[vertex](a) at (25, 0) {\LARGE $1$};
\node[vertex](b) at (31, 0) {\LARGE $1$};
\node[vertex](c) at (28, 2.5) {\LARGE $1$};
\node[vertex](d) at (28, 6) {\LARGE $1$};

\begin{scope}[every path/.style={-}]
\draw (a) -- node[midway, fill=white] {\LARGE $1$} (d);
\draw (b) -- node[midway, fill=white] {\LARGE $1$} (d);
\draw (c) -- node[midway, fill=white] {\LARGE $1$} (d);
\end{scope}

\end{tikzpicture}

\caption{Decomposition in $\agg(\mathcal K_4)$ illustrating $\rho(\agg(\cK_n)) = n + \tfrac{2}{n} - 2$.
See \cref{exm:elasticity}. }  
 \label{fig:k4}
\end{figure}

Non-simple graphs can also attain the weaker upper bound in \ref{rho:generic}; see \cref{exm:bk} below.
We are also able to give a bound on the refined elasticities.

\begin{theorem} \label{thm:rhok}
    Let $\cG$ be a graph.
    \begin{enumerate}
        \item \label{rho:refined-bound} For all $k \ge 2$,
        \[  
        \rho_k(\agg(\cG)) \le (k-1)\card{V} + 1.
        \]
        \item \label{rho:equality}
        Equality in the previous bound holds if and only if $\cG$ contains $k$ edge-disjoint spanning trees.
        In particular, the graph $\cG$ must have edge-connectivity at least $k$.
    \end{enumerate}
\end{theorem}

\begin{proof}
\ref{rho:refined-bound}
Let $a \in \agg(\cG)$ with $k \in \sL(a)$.
Suppose first that $\supp(a)$ is connected.
By \cref{l:substructure} we may assume $\supp(a)=\cG$.
Let $\cT$ be a spanning tree of $\cG$.
Since $k \in \sL(a)$, we have $\sum_{v \in V} a(v) \le k \card{V}$.

Let now $a = b_1 + \cdots + b_l$ with atoms $b_1$, $\ldots\,$,~$b_l$.
For each $i \in [1,l]$ let $n_i(\cT) \in \bN_0$ be the number of edges of the spanning tree $\cT$ that are contained in $\supp(b_i)$.
Observe that $\supp(b_i)$ must contain at least $n_i(\cT)+1$ vertices.
Moreover $\sum_{i=1}^l n_i(\cT) \ge \card{V} - 1$.
Thus
\begin{equation}\label{eq:ineq-rhok}
l + \card{V} - 1 \le \sum_{i=1}^{l} (n_i(\cT) + 1) \le \sum_{i=1}^l \sum_{v \in V} b_i(v) = \sum_{v \in V} a(v) \le k \card{V},
\end{equation}
and therefore $l \le (k-1)\card{V} + 1$.

Now suppose that $\supp(a)$ is arbitrary, and that $a=c_1+ \cdots + c_k = b_1+\cdots + b_l$ for atoms $c_i$,~$b_i$ with $l \ge k$.
Let $\cG_1=(V_1,E_1,r_1)$, $\ldots\,$,~$\cG_s=(V_s,E_s,r_s)$ be the connected components of $\supp(a)$.
Again reducing to $\supp(a)=\cG$, we have $\agg(\cG) \cong \agg(\cG_1) \times \cdots \times \agg(\cG_s)$.
If $\card{V_i}=1$ for some $i$, the atom $\ind_{\cG_i}$ must appear with the same multiplicity in every factorization of $a$.
Thus, for bounding $l/k$, we may assume without restriction $\card{V_i} \ge 2$ for every $i \in [1,s]$.
Moreover $2 \le s \le k$.
Then $\card{V_i} \le \card{V} - 2$ for every $i \in [1,s]$.

For every $i \in [1,s]$, let $k_i = \card{\{\, j \in [1,k] : \supp(c_j) \subseteq \cG_i \,\}}$ and $l_i = \card{\{\, j \in [1,l] : \supp(b_j) \subseteq \cG_i \,\}}$.
Then we know $l_i \le (k_i - 1)\card{V_i} + 1$.
Thus
\[
\begin{split}
l &= \sum_{i=1}^s l_i \le s + \sum_{i=1}^s (k_i -1) \card{V_i} \le s + \sum_{i=1}^s (k_i-1) (\card{V}-2)  
  \le s + (k-s) (\card{V}-2) \\
  &\le s + (k-1) (\card{V} - 2) =  s + (k-1)\card{V} - 2(k-1) \le (k-1)\card{V},
\end{split}
\]
where the last inequality follows from $2 \le s \le k$.
Note that this bound is strictly smaller than the one we claim for $\rho_k(\agg(\cG))$.
This implies that any $a$ achieving the upper bound $\rho_k(a) = (k-1)\card{V}+1$ must have connected support with $\supp(a)=\cG$.

\ref{rho:equality}
Let $k\ge 2$.
Suppose there exists $a \in \agg(\cG)$ and atoms $c_1$, $\ldots\,$,~$c_k$, $b_1$, $\ldots\,$,~$b_l$ with $l =(k-1)\card{V} + 1$ such that
\[
a=c_1 + \cdots + c_k = b_1 + \cdots + b_l.
\]
By what we just showed $\supp(a)=\cG$ and $\cG$ is connected.
Without restriction we may replace each $c_i$ with the indicator of a spanning tree of $\supp(c_i)$.
This may change $a$ and the $b_i$, but cannot make $l$ smaller.

Let $\cT$ be a spanning tree of $\cG$, and as before, denote by $n_i(\cT)$ the number of edges of $\cT$ that appear in $\supp(b_i)$.
Then in \cref{eq:ineq-rhok} we must have equality throughout, so that 
\[
\sum_{i=1}^{l} (n_i(\cT) + 1) = k \card{V} = \sum_{i=1}^k \sum_{v \in V} c_i(v).
\]
The first thing we observe from this is that necessarily each $c_i$ must have $\card{V}$ vertices, so that $\supp(c_i)$ is a spanning tree.
Secondly, suppose that distinct $c_i$ and $c_j$ share an edge.
Taking $\cT = \supp(c_i)$, then $\sum_{i=1}^{l} (n_i(\cT) + 1)\ge l + \card{V}$, and thus $l \le (k-1) \card{V}$, a contradiction.
We have thus shown that $\supp(c_1)$, $\ldots\,$,~$\supp(c_k)$ are pairwise edge-disjoint spanning trees of $\cG$.
\end{proof}

\begin{remark}
    The \defit{spanning tree packing number $\tau(\cG)$} of a graph $\cG$ is the largest number $k$ such that $\cG$ contains $k$ edge-disjoint spanning trees.
    By the previous theorem, the number $\tau(\cG)$ is the precise $k$ at which the value of $\rho_k(\agg(\cG))$ starts to differ from the upper bound.
    The survey \cite{palmer01} discusses several results about the spanning tree packing number.
    We also mention the more recent papers \cite{liu-hong-gu-lai14,gao-perezgimenez-sato18,liu-lai-tian19}.
\end{remark}

\begin{example} \label{exm:bk}
    Let $\cB_k$ be the graph consisting of two vertices and $k \ge 2$ edges between them.
    Since $\cB_k$ has $k$ edge-disjoint spanning trees we have $\rho_{k'}(\agg(\cB_k)) = (k'-1)\card{V} +1$ for $2 \le k' \le k$.
    Thus $\rho(\agg(\cB_k)) \ge \frac{k-1}{k} \card{V} + \frac{1}{k} = 2 - \frac{1}{k}$.
    This shows that the bound from \ref{rho:generic} of \cref{thm:rho} can be attained.
\end{example}

We now concern ourselves with the question when $\agg(\cG)$ is half-factorial and when it is factorial.

\begin{lemma} \label{l:hf}
    Let $\cG=(V,E,r)$ be a graph.
    If $a \in \agg(\cG)$ and $\supp(a)$ is acyclic, then $\sL(a) = \{l\}$ with
    \[
        l = \sum_{v \in V} a(v) - \sum_{e \in E} a(e).
    \]
    In particular $\card{\sL(a)}=1$.
\end{lemma}

\begin{proof}
Using \cref{l:substructure} it suffices to consider the case where $\cG$ is a tree.
Suppose $b$ is a summand of $a$. Then $\supp(b)$ is a subgraph of $\supp(a)$ and therefore acyclic.
In particular, $\supp(b)$ is a tree for every atom $b$ dividing $a$. 
For a tree $\cT=(V',E',r)$ we have $\card{V'}-\card{E'}=1$.
Thus, if $a = b_1 + \cdots + b_l$ with $b_1$,~$\ldots\,$,~$b_l$ indicators of trees $\cT_i=(V_i,E_i,r_i)$, we must have
\[
\sum_{v \in V} a(v) - \sum_{e \in E} a(e) = \sum_{i=1}^l \big( \card{V_i} - \card{E_i} \big) = l. \qedhere
\]
\end{proof}

\begin{remark}
The converse of the previous lemma is false: there exist agglomerations $a$ with $\supp(a)$ cyclic and $\card{\sL(a)}=1$. Since, by \cref{l:nicefact}, all atoms are absolutely irreducible, this is easily seen by considering $na$ where $n$ is any positive integer and $a$ is any atom in $\agg(\cG)$ that involves a cycle. For a more interesting example, let $\cG$ be the complete graph on five vertices and let $\cG_1$ and $\cG_2$ be the full subgraphs of $\cG$ with vertex sets $\{v_1, v_2, v_3\}$ and $\{v_3, v_4, v_5\}$, respectively. Then the only other factorization of $a=\ind_{\cG_1}+\ind_{\cG_2}$ is $a=(a-v_3)+v_3$. Consequently, $\sL(a)=\{2\}$, yet $a$ contains a cycle.
\end{remark}

\begin{theorem} \label{t:factorial}
Let $\cG=(V,E,r)$ be a graph.
\begin{enumerate}
    \item\label{factorial:hf} The monoid $\agg(\cG)$ is half-factorial if and only if $\cG$ is acyclic.
    \item\label{factorial:f} The monoid $\agg(\cG)$ is factorial if and only if every connected component of $\cG$ contains at most one edge.
\end{enumerate}
\end{theorem}

\begin{proof}
\ref{factorial:hf}
    If $\cG$ is acyclic, then $\agg(\cG)$ is half-factorial by the \cref{l:hf}.
    
    Suppose now that $\cG$ is not acyclic.
    Then $\cG$ contains a simple $n$-cycle with $n \ge 3$ or a multiple edge (that is, a two-cycle).
    If $\cG$ contains a simple $n$-cycle ($n \ge 3$), then $\agg(\cC_n)$ embeds as a divisor-closed submonoid into $\cG$ and hence $\agg(\cG)$ is not half-factorial by \cref{exm:elasticity}.
    If $\cG$ contains a two-cycle, then $\agg(\cB_2)$ embeds as a divisor-closed submonoid into $\agg(\cG)$ and hence $\agg(\cG)$ is not half-factorial, again by \cref{exm:bk}.
    
\ref{factorial:f}   
    The characterization of factorial agglomeration monoids follows from \ref{nicefact:primes} of \cref{l:nicefact} or \cref{c:class-group}.
\end{proof}

We have focused on the study of the (refined) elasticities of $\agg(\cG)$, as these are the most basic arithmetical invariants.
To finish, we give one example involving other invariants, demonstrating that large distances and catenary degrees can occur.

\begin{example} \label{exm:k2n}
    The following construction is illustrated in \cref{fig:k2n}.
    Let $\cG=\cK_{2,n}$ with $n \ge 2$ be a complete bipartite graph, with vertices $v_1$,~$v_2$, and $w_1$, $\ldots\,$,~$w_n$, and an edge from $v_i$ to $w_j$ for all $i \in [1,2]$ and $j \in [1,n]$.
    Let $\cG_i$ be the full subgraph on $v_i$, $w_1$, $\ldots\,$,~$w_n$, and let $a = \ind_{\cG_1} + \ind_{\cG_2}$.
    Then $a(e)=1$ for all edges $e$.
    On the vertices $a(w_j)=2$ for $j \in [1,n]$ and $a(v_i)=1$ for $i \in [1,2]$.
    The splitting lemma gives $a = \ind_{\cG} + \sum_{j=1}^n \ind_{w_j}$.
    We claim that these are the only two factorizations of $a$.
    
    Indeed, suppose $b \in \agg(\cG)$ is an atom dividing $a$ with $v_i \in \supp(b)$ for some $i \in [1,2]$.
    Then necessarily $b(e)=1$ for every edge from $v_i$ to any $w_j$.
    Thus $\cG_i$ is a subgraph of $\supp(b)$.
    So either $b=\ind_{\cG_i}$ or $b=\ind_{\cG}$.
    The claim is an immediate consequence of this.
    
    The set of lengths of $a$ is $\sL(a) = \{2,n+1\}$. 
    Thus $n-1 \in \Delta(\agg(\cK_{2,n}))$ for the set of distances of $\agg(\cK_{2,n})$.
    As a consequence, the catenary degree of $\agg(\cK_{2,n})$ is at least $n+1$.
\end{example}

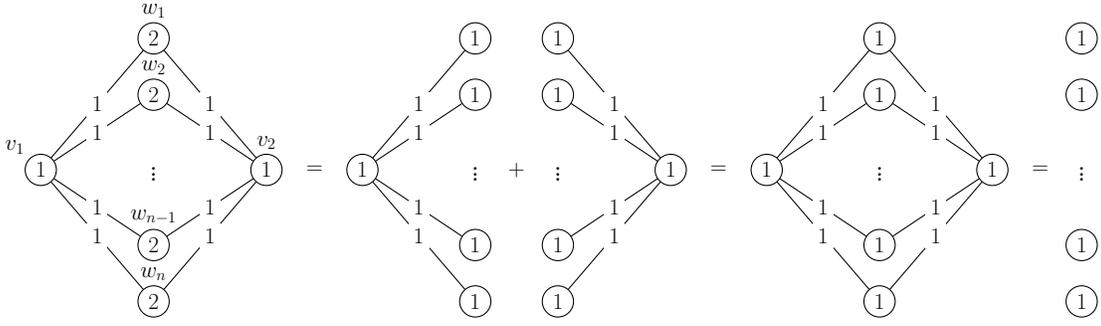
\begin{figure}
\centering

\tikzstyle{vertex}=[circle, draw, anchor=center]

\noindent \begin{tikzpicture}[transform shape, scale=0.5]
\node[vertex](l) at (0, 7) [label=135:\LARGE $v_1$] {\LARGE $1$};
\node[vertex](w1) at ($(l)+(3,3.5)$) [label={\LARGE $w_1$}]  {\LARGE $2$};
\node[vertex](w2) at ($(w1)-(0,1.5)$) [label={\LARGE $w_2$}]  {\LARGE $2$};
\node[vertex](r) at ($(l)+(6,0)$) [label={\LARGE $v_2$}] {\LARGE $1$};
\node[vertex](wn) at ($(l)+(3,-3.5)$)  [label={\LARGE $w_{n}$}] {\LARGE $2$};
\node[vertex](wn1) at ($(wn)+(0,1.5)$) [label={\LARGE $w_{n-1}$}] {\LARGE $2$};
\node at ($(l)+(3,0)$) {\Huge $\vdots$};

\begin{scope}[every path/.style={-}, every node/.style={fill=white,midway}]
  \draw (l) -- (w1) node {\LARGE $1$};
  \draw (l) -- (w2) node {\LARGE $1$};
  \draw (l) -- (wn1) node {\LARGE $1$};
  \draw (l) -- (wn) node {\LARGE $1$};

  \draw (w1) -- (r) node {\LARGE $1$};
  \draw (w2) -- (r) node {\LARGE $1$};
  \draw (wn1) -- (r) node {\LARGE $1$};
  \draw (wn) -- (r) node {\LARGE $1$};
\end{scope}

%%%%%%%%%%%%%%%%%%%%%%%%% ===== %%%%%%%%%%%%%%%%%%%%%%%%%%%
\node (eq1) [right=0.5 of r] {\LARGE $=$};

\node[vertex](l) [right=0.5 of eq1] {\LARGE $1$};
\node[vertex](w1) at ($(l)+(3,3.5)$) {\LARGE $1$};
\node[vertex](w2) at ($(w1)-(0,1.5)$) {\LARGE $1$};
\node[vertex](wn) at ($(l)+(3,-3.5)$) {\LARGE $1$};
\node[vertex](wn1) at ($(wn)+(0,1.5)$) {\LARGE $1$};
\node(dots) at ($(l)+(3,0)$) {\Huge $\vdots$};

\begin{scope}[every path/.style={-}, every node/.style={fill=white,midway}]
  \draw (l) -- (w1) node {\LARGE $1$};
  \draw (l) -- (w2) node {\LARGE $1$};
  \draw (l) -- (wn1) node {\LARGE $1$};
  \draw (l) -- (wn) node {\LARGE $1$};
\end{scope}

%%%%%%%%%%%%%%%%%%%%%%%%% +++++++ %%%%%%%%%%%%%%%%%%%%%%%%%%%
\node (eq2) [right=0.5 of dots] {\LARGE $+$};

\node(dots) [right=0.5 of eq2] {\Huge $\vdots$};
\node[vertex](w1) at ($(dots)+(0,3.5)$) {\LARGE $1$};
\node[vertex](w2) at ($(w1)-(0,1.5)$) {\LARGE $1$};
\node[vertex](r) at ($(dots)+(3,0)$) {\LARGE $1$};
\node[vertex](wn) at ($(dots)+(0,-3.5)$) {\LARGE $1$};
\node[vertex](wn1) at ($(wn)+(0,1.5)$) {\LARGE $1$};

\begin{scope}[every path/.style={-}, every node/.style={fill=white,midway}]
  \draw (w1) -- (r) node {\LARGE $1$};
  \draw (w2) -- (r) node {\LARGE $1$};
  \draw (wn1) -- (r) node {\LARGE $1$};
  \draw (wn) -- (r) node {\LARGE $1$};
\end{scope}

%%%%%%%%%%%%%%%%%%%%%%%%% ===== %%%%%%%%%%%%%%%%%%%%%%%%%%%
\node (eq3) [right=0.5 of r] {\LARGE $=$};

\node[vertex](l) [right=0.5 of eq3] {\LARGE $1$};
\node[vertex](w1) at ($(l)+(3,3.5)$) {\LARGE $1$};
\node[vertex](w2) at ($(w1)-(0,1.5)$) {\LARGE $1$};
\node[vertex](r) at ($(l)+(6,0)$) {\LARGE $1$};
\node[vertex](wn) at ($(l)+(3,-3.5)$) {\LARGE $1$};
\node[vertex](wn1) at ($(wn)+(0,1.5)$) {\LARGE $1$};
\node at ($(l)+(3,0)$) {\Huge $\vdots$};

\begin{scope}[every path/.style={-}, every node/.style={fill=white,midway}]
  \draw (l) -- (w1) node {\LARGE $1$};
  \draw (l) -- (w2) node {\LARGE $1$};
  \draw (l) -- (wn1) node {\LARGE $1$};
  \draw (l) -- (wn) node {\LARGE $1$};

  \draw (w1) -- (r) node {\LARGE $1$};
  \draw (w2) -- (r) node {\LARGE $1$};
  \draw (wn1) -- (r) node {\LARGE $1$};
  \draw (wn) -- (r) node {\LARGE $1$};
\end{scope}

%%%%%%%%%%%%%%%%%%%%%%%%% +++++ %%%%%%%%%%%%%%%%%%%%%%%%%%%
\node (eq4) [right=0.5 of r] {\LARGE $=$};

\node(dots) [right=0.5 of eq4] {\Huge $\vdots$};
\node[vertex](w1) at ($(dots)+(0,3.5)$) {\LARGE $1$};
\node[vertex](w2) at ($(w1)-(0,1.5)$) {\LARGE $1$};
\node[vertex](wn) at ($(dots)+(0,-3.5)$) {\LARGE $1$};
\node[vertex](wn1) at ($(wn)+(0,1.5)$) {\LARGE $1$};

\end{tikzpicture}

\caption{In $\agg(\mathcal K_{2,n})$ there exists an element $a$ with set of lengths $\sL(a)=\{2,n+1\}$.
See \cref{exm:k2n}.}   \label{fig:k2n}
\end{figure}

\section{From lattices to graph agglomerations} \label{s:tie-in}

In this final section we bridge the studies of lattices of modules over Bass rings and monoids of graph agglomerations.
By establishing a transfer homomorphism, we are able to prove \cref{t:main,t:main-finiteness}.

\begin{definition} \label{d:prime-ideal-intersection-graph}
    Let $R$ be a Bass ring.
    The \defit{graph of prime ideal intersections} of $R$, denoted by $\cG_R$, is the graph with
    \begin{itemize}
        \item set of vertices equal to $\minspec(R)$, 
        \item set of edges consisting of all maximal ideals of $R$ that contain more than one minimal prime ideal, and
        \item each edge being incident with the two minimal prime ideals contained in it.
    \end{itemize}
\end{definition}

Since every maximal ideal in a Bass ring $R$ contains at most two minimal prime ideals, this is a well defined graph (possibly with multiple edges, but with no loops).
Because $R$ has finitely many minimal prime ideals and the set of singular maximal ideals is finite, the graph is finite.
We discuss two examples below (\cref{exm:bass-agg}). 

\begin{remark}
    Let $R$ be a Bass ring with minimal prime ideals $\fp_1$, $\ldots\,$,~$\fp_k$.
    The Zariski closed set $V_i\coloneqq \overline{ \{\fp_i\} } = \{\, \fq \in \Spec(R) : \fp_i \subseteq \fq \,\}$ is open, as it is the complement of the closed set $\bigcup_{j \ne i} V_j$.
    Two distinct such sets $V_i$ and $V_j$ have non-trivial intersection if and only if there exists a maximal ideal $\fm$ containing both $\fp_i$ and $\fp_j$.
    Therefore $\cG_R$ is the intersection (multi)graph of the sets $V_1$, $\ldots\,$,~$V_k$ in the graph-theoretical sense.
        
    We also see that $R$ is indecomposable as a ring if and only if $\cG_R$ is connected.
    If $R= R_1 \times \cdots \times R_l$ with indecomposable rings $R_1$, $\ldots\,$,~$R_l$, then $\cG_R  \cong \cG_{R_1} \oplus \cdots \oplus \cG_{R_l}$, with $\cG_{R_i}$ the connected components of the graph $\cG_R$.
\end{remark}

We now recognize the codomain of the transfer homomorphism of \cref{p:transfer-deduplicated} as a monoid of graph agglomerations.

\begin{proposition} \label{p:transfer-bass-agg}
    Let $R$ be a Bass ring and let $\cG_R$ be its prime ideal intersection graph.
    Then there exists a transfer homomorphism $\theta\colon T(R) \to \agg(\cG_R)$ from the monoid of isomorphism classes of $R$-lattices to the monoid of agglomerations on $\cG_R$.
\end{proposition}

\begin{proof}
    If $R=0$, then $\cG_R$ is the null graph, and $T(R)=\agg(\cG_R)=0$.
    If $\dim(R)=0$, then $R=K_1 \times \cdots \times K_k$ is a product of fields and $\cG_R$ is a disjoint union of $k$ trivial graphs.
    Hence $T(R) \cong \bN_0^k \cong \agg(\cG_R)$ and there is nothing to show.
    From now on $\dim(R)=1$.

    Let $H \subseteq \bN_0^c$ be the Diophantine monoid from \cref{p:transfer-deduplicated}.
    With $E \subseteq \maxspec(R)$ the set of maximal ideals of $R$ containing two minimal prime ideals, with $V \coloneqq \minspec(R)$, and with indexing as in \cref{p:transfer-deduplicated}, we have
    \[
    H = \{\, \vec x \in \bN_0^c : x_\fp = x_\fm + x_{\fm_{\fp}} \text{ for all $(\fm,\fp) \in E \times V$ with $\fp \subseteq \fm$ } \,\}.
    \]    

    Since there exists a transfer homomorphism $\varphi\colon T(R) \to H$, it is sufficient to show $H \cong \agg(\cG_R)$.
    Note that $x_\fp \ge x_\fm$ for every $(\fm,\fp) \in E$ with $\fp \subseteq \fm$, and that $x_{\fm_{\fp}}$ can then be computed as $x_{\fm_\fp} = x_\fp - x_{\fm}$. 
    Dropping the coordinates corresponding to columns of the type $\fm_\fp$, we have that
    \[
    H \cong \{\, \vec x \in \bN_0^{V \cup E} : x_{\fp} \ge x_{\fm} \text{ for all $(\fp,\fm) \in E$ with $\fp \subseteq \fm $} \,\}.
    \]
    Then $H \cong \agg(\cG_R)$ by definition of the monoid of graph agglomerations, \cref{def:agglomeration}.
\end{proof}

Not only can the arithmetic of $T(R)$ be studied in terms of a monoid of graph agglomerations whenever $R$ is a Bass ring, but every graph can be realized as the graph of prime ideal intersections of some Bass ring.  In the following realization result, keep in mind that all our graphs are finite and that we permit multiple edges but no loops.

\begin{proposition} \label{p:realization}
\begin{enumerate}
    \item \label{realization:subgraph} Let $R$ be a Bass ring and let $\cG'$ be a subgraph of $\cG_R$.
    Then there exists a multiplicative set $S \subseteq R$ such that the localization $R' = (S^{-1})R$ satisfies $\cG_{R'}\cong \cG'$.
    \item \label{realization:all} For every finite graph $\cG$, there exists a semilocal Bass ring $R$ with $\cG_R\cong \cG$.
\end{enumerate}
\end{proposition}

\begin{proof}
    \ref{realization:subgraph}
    Let $\fm_1$, $\ldots\,$,~$\fm_k$ be the maximal ideals of $R$ that are the edges of $\cG'$, and let $\fp_1$, $\ldots\,$,~$\fp_l$ denote the minimal prime ideals of $R$ that are the vertices of $\cG'$.
    Then $S \coloneqq R \setminus (\fp_1 \cup \cdots \cup \fp_l \cup \fm_1\cup\cdots\cup \fm_k)$ is a multiplicative set. 
    The localization $R'=S^{-1}R$ is a Bass ring with minimal prime ideals $R'\fp_1,\ldots,R'\fp_l$ and maximal ideals $R'\fm_1,\ldots,R'\fm_k$.
    (In the degenerate case $l=0$, the ring $R'$ is the zero ring.)
    Since the inclusions between these ideals are preserved, we have $\cG_R' \cong \cG'$.
    
    \ref{realization:all} 
    The null graph is realized by the zero ring. 
    Assume $\cG$ is non-null.
    By \ref{realization:subgraph} it suffices to prove: for all $n \ge 1$ and $d \ge 1$, there exists a semilocal Bass ring $R$ with $\cG_R$ having $n$ vertices, and at least $d$ edges between any two vertices.
    Let $K$ be an algebraically closed field and let $C_1$, $\ldots\,$,~$C_n$ be pairwise distinct, smooth affine plane curves over $K$ such that any two of them have at least $d$ intersection points.
    We further assume that all these intersection points are pairwise distinct.
    Smoothness guarantees that every $C_i$ has multiplicity $1$ at each of its points.
    Thus $C = C_1 \cup \cdots \cup C_n$ has multiplicity $2$ at each point where two curves intersect, and multiplicity $1$ at every other point.
    It follows that the affine coordinate ring $R=K[C]$ is a Bass ring, and that $\cG_R$ has $n$ vertices and at least $d$ edges between any two vertices.
    
    Let $\fm_1$, $\ldots\,$,~$\fm_k$ denote the maximal ideals of $R$ corresponding to the intersection points of the curves $C_1$, $\ldots\,$,~$C_n$.
    Let $S \coloneqq R \setminus (\fm_1 \cup \cdots \cup \fm_k)$.
    Then $S^{-1}R$ is a semilocal Bass ring with $\cG_{S^{-1}R} \cong \cG_R$.
\end{proof}

\begin{remark}
    If $\cG$ is a simple graph, then we can take the curves $C_1$, $\ldots\,$,~$C_n$ in the proof of \ref{realization:subgraph} to be $n$ lines in general position in the affine plane (that is, no two lines are parallel, and no three lines intersect in a point). 
    
    For higher multiplicity of the edges, it is similarly easy to construct higher degree curves with enough intersection points. 
    For instance, one could take a polynomial $f(x) \in K[x]$ of degree $d+1$.
    For any $\lambda\in K$, the graph $C_\lambda$ defined by  $y=f(x - \lambda)$ is a smooth plane curve.
    If $\lambda_1 \ne \lambda_2$, then $C_{\lambda_1}$ intersects $C_{\lambda_2}$ in $d$ points (counting multiplicity).
    Choosing among the infinitely many curves of this family $n$ of them having no multiple intersection points yields the desired family.
 \end{remark}

\Cref{t:main,t:main-finiteness} are now an easy consequence of the existence of a transfer homomorphism together with the corresponding results for monoids of graph agglomerations.

\begin{proof}[Proof of \cref{t:main}]
   The transfer homomorphism exists by \cref{p:transfer-bass-agg}.
   Since $\agg(\cG_R)$ is a finitely generated Krull monoid, therefore $T(R)$ is transfer Krull of finite type.
   The remaining claims follow from \cref{p:transfer-implication}.
\end{proof}

\begin{proof}[Proof of \cref{t:main-finiteness}]
    By \cref{t:main}, the monoid $T(R)$ is transfer Krull of finite type.
    More specifically, there is a transfer homomorphism $\theta\colon T(R) \to \agg(\cG_R)$.
    Claims \ref{mf:elasticities-general}, \ref{mf:distances}, \ref{mf:uk}, and \ref{mf:lengths} thus follow from \cref{t:krull-finiteness}; and   \ref{mf:elasticities} and \ref{mf:refined-elasticities} follow from \cref{thm:rho} and \cref{thm:rhok}; claim \ref{mf:half-factorial} holds by \cref{t:factorial}.
    
    If $\Pic(R)$ is trivial, the map $\psi$ from \cref{t:transhom} is injective.
    Then $T(R)\cong \Psi(T(R))$ is a finitely generated Krull monoid and \ref{mf:cat-omega} follows from \cref{t:krull-finiteness}.

   For the complete graph $\cK_n$, recall $\rho(\agg(\cK_n)) = n + \frac{2}{n} - 2 \ge n-2$ by \cref{exm:elasticity}.
    By \cref{p:realization} there is a Bass Ring $R$ with $\rho(T(R)) = \rho(\agg(\cK_n)) \ge n-2$; this shows \ref{mf:existence}.
\end{proof}

We now revisit \cref{exm:bass-domain,exm:bass-ngon,exm:k2n} in light of \cref{p:realization}.
\begin{example} \label{exm:bass-agg}
\begin{enumerate}
    \item As in \cref{exm:bass-domain}, let $R$ be a domain.
    Then $R$ has the unique minimal prime ideal $0$.
    Thus $\cG_R$ is the trivial graph consisting of a single vertex and $\agg(\cG_R) \cong \bN_0$.
    Consequently, the monoid $\agg(\cG_R)$ is factorial, and therefore $T(R)$ is half-factorial.
    
    \item Let $R$ be the localization of the coordinate ring of a regular $m$-gon, as in \cref{exm:bass-ngon}.
    Then there are $m$ minimal prime ideals, corresponding to the lines containing the edges of the $m$-gon, and $m$ maximal ideals corresponding to the corners of the $m$-gon. 
    The graph $\cG_R$ is therefore an $m$-cycle (but keep in mind that the edges correspond to the maximal ideals, and the vertices correspond to the maximal ideals here).
    Thus, for instance, we have $\rho(T(R)) = 2 - \frac{1}{m}$ by \cref{exm:elasticity}.
    
    \item For each $n\geq 2$, \cref{p:realization} and \cref{exm:k2n} give the existence of a semilocal Bass ring $R$ and an $R$-lattice $M$ such that $M$ decomposes only as the direct sum of $2$ indecomposable lattices and as the direct sum of $n+1$ indecomposable lattices. Consequently $\Delta([M])=\{n-1\}$ in $T(R)$. We contrast this with the case where $R$ is a local ring-order with finite representation type where  $\Delta([M])=\{0\}$ or $\Delta([M])=\{1\}$ for all $M$ (see \cite[Theorems 4.12 and 6.4(a)]{baeth-geroldinger14}). 
\end{enumerate}
\end{example}

\begin{remark}
    \begin{enumerate}
        \item The transfer homomorphism $\psi$ in \cref{subsec:diophantine} maps the isomorphism class of a module $M$ to its genus; for each singular maximal ideal we record the multiplicity of the indecomposable modules in the direct-sum decomposition $M_\fm$.
        If $\fm$ contains two minimal prime ideals, each of these modules $M_\fm$ has rank $(1,0)$, $(0,1)$, or $(1,1)$.
        In passing to the monoid of graph agglomerations we lose some additional information: only the ranks of the indecomposable summands of $M_\fm$ are recorded (this corresponds to the elimination of duplicate columns in the defining matrix of the Diophantine monoid, as in \cref{l:duplicate columns}).
        Put another way, if there are multiple modules of rank $(1,1)$, we do not distinguish between them.
        
        To correct for this loss of information, it would be possible to refine the monoid of graph agglomerations as follows: 
        put an edge for each $\fm$ containing two minimal prime ideals and each indecomposable $N_\fm$ and group these edges by $\fm$.
        Then the function associating values to the edges and vertices of the graphs must be defined so that the value of each vertex is larger than or equal to the sum of weights of all the edges corresponding to $\fm$.
        Although this definition would provide a more complete translation of the genus (it still does not account for the different indecomposables at a singular maximal ideal containing a unique minimal prime), such a monoid would be more complicated to study.

        Since there is already a transfer homomorphism from $T(R)$ to a monoid of graph agglomerations, this additional complexity does not provide more information as far as sets of lengths are concerned.
     
        \item 
        More refined arithmetical invariants, such as the tame degree, catenary degree, or $\omega$-invariant, are in general not preserved by transfer homomorphisms. 
        To pull back information about these invariants along a transfer homomorphism, one needs to study the fibers of the transfer homomorphism.
        In the present paper, we have chosen to forego the additional technical complexity this entails.
        For the very special case in which $R$ is a Dedekind domain, the results from \cite[Sections 5 and 6]{baeth-geroldinger-grynkiewicz-smertnig15} are applicable.
    \end{enumerate}
\end{remark}

\bibliographystyle{hyperalphaabbr}
\bibliography{bass_rings}

\begin{thebibliography}{FHKW06}

\bibitem[AA92]{anderson-anderson92}
D.~D. Anderson and D.~F. Anderson.
\newblock Elasticity of factorizations in integral domains.
\newblock {\em J. Pure Appl. Algebra}, 80(3):217--235, 1992.
\newblock \href {http://dx.doi.org/10.1016/0022-4049(92)90144-5}
  {\path{doi:10.1016/0022-4049(92)90144-5}}.

\bibitem[And97]{anderson97}
D.~D. Anderson, editor.
\newblock {\em Factorization in integral domains}, volume 189 of {\em Lecture
  Notes in Pure and Applied Mathematics}, New York, 1997. Marcel Dekker Inc.

\bibitem[Bae07]{Baeth-07}
N.~R. Baeth.
\newblock A {K}rull-{S}chmidt theorem for one-dimensional rings of finite
  {C}ohen-{M}acaulay type.
\newblock {\em J. Pure Appl. Algebra}, 208(3):923--940, 2007.
\newblock \href {http://dx.doi.org/10.1016/j.jpaa.2006.03.023}
  {\path{doi:10.1016/j.jpaa.2006.03.023}}.

\bibitem[Bas63]{bass63}
H.~Bass.
\newblock On the ubiquity of {G}orenstein rings.
\newblock {\em Math. Z.}, 82:8--28, 1963.
\newblock \href {http://dx.doi.org/10.1007/BF01112819}
  {\path{doi:10.1007/BF01112819}}.

\bibitem[BG14]{baeth-geroldinger14}
N.~R. Baeth and A.~Geroldinger.
\newblock Monoids of modules and arithmetic of direct-sum decompositions.
\newblock {\em Pacific J. Math.}, 271(2):257--319, 2014.
\newblock \href {http://dx.doi.org/10.2140/pjm.2014.271.257}
  {\path{doi:10.2140/pjm.2014.271.257}}.

\bibitem[BGGS15]{baeth-geroldinger-grynkiewicz-smertnig15}
N.~R. Baeth, A.~Geroldinger, D.~J. Grynkiewicz, and D.~Smertnig.
\newblock A semigroup-theoretical view of direct-sum decompositions and
  associated combinatorial problems.
\newblock {\em J. Algebra Appl.}, 14(2):1550016, 60, 2015.
\newblock \href {http://dx.doi.org/10.1142/S0219498815500164}
  {\path{doi:10.1142/S0219498815500164}}.

\bibitem[BGR20]{bashir-geroldinger-reinhart20}
A.~Bashir, A.~Geroldinger, and A.~Reinhart.
\newblock On the arithmetic of stable domains.
\newblock 2020.
\newblock Preprint.
\newblock \href {http://arxiv.org/abs/2007.05574} {\path{arXiv:2007.05574}}.

\bibitem[BL11]{baeth-luckas11}
N.~R. Baeth and M.~R. Luckas.
\newblock Monoids of torsion-free modules over rings with finite representation
  type.
\newblock {\em J. Commut. Algebra}, 3(4):439--458, 2011.
\newblock \href {http://dx.doi.org/10.1216/JCA-2011-3-4-439}
  {\path{doi:10.1216/JCA-2011-3-4-439}}.

\bibitem[BS12]{baeth-saccon12}
N.~R. Baeth and S.~Saccon.
\newblock Monoids of modules over rings of infinite {C}ohen-{M}acaulay type.
\newblock {\em J. Commut. Algebra}, 4(3):297--326, 2012.
\newblock \href {http://dx.doi.org/10.1216/jca-2012-4-3-297}
  {\path{doi:10.1216/jca-2012-4-3-297}}.

\bibitem[BW13]{baeth-wiegand13}
N.~R. Baeth and R.~Wiegand.
\newblock Factorization theory and decompositions of modules.
\newblock {\em Amer. Math. Monthly}, 120(1):3--34, 2013.
\newblock \href {http://dx.doi.org/10.4169/amer.math.monthly.120.01.003}
  {\path{doi:10.4169/amer.math.monthly.120.01.003}}.

\bibitem[CFGO16]{chapman-fontana-geroldinger-olberding16}
S.~T. Chapman, M.~Fontana, A.~Geroldinger, and B.~Olberding, editors.
\newblock {\em {M}ultiplicative {I}deal {T}heory and {F}actorization {T}heory}.
  Springer International Publishing, 2016.
\newblock Commutative and {N}on-commutative {P}erspectives.
\newblock \href {http://dx.doi.org/10.1007/978-3-319-38855-7}
  {\path{doi:10.1007/978-3-319-38855-7}}.

\bibitem[Cha05]{chapman05}
S.~T. Chapman, editor.
\newblock {\em Arithmetical properties of commutative rings and monoids},
  volume 241 of {\em Lecture Notes in Pure and Applied Mathematics}. Chapman \&
  Hall/CRC, Boca Raton, FL, 2005.
\newblock \href {http://dx.doi.org/10.1201/9781420028249}
  {\path{doi:10.1201/9781420028249}}.

\bibitem[CWW95]{cimen-wiegand-wiegand95}
N.~Cimen, R.~Wiegand, and S.~Wiegand.
\newblock One-dimensional rings of finite representation type.
\newblock In {\em Abelian groups and modules ({P}adova, 1994)}, volume 343 of
  {\em Math. Appl.}, pages 95--121. Kluwer Acad. Publ., Dordrecht, 1995.

\bibitem[Dir07]{diracca07}
L.~Diracca.
\newblock On a generalization of the exchange property to modules with
  semilocal endomorphism rings.
\newblock {\em J. Algebra}, 313(2):972--987, 2007.
\newblock \href {http://dx.doi.org/10.1016/j.jalgebra.2007.02.041}
  {\path{doi:10.1016/j.jalgebra.2007.02.041}}.

\bibitem[Fac02]{facchini02}
A.~Facchini.
\newblock Direct sum decompositions of modules, semilocal endomorphism rings,
  and {K}rull monoids.
\newblock {\em J. Algebra}, 256(1):280--307, 2002.
\newblock \href {http://dx.doi.org/10.1016/S0021-8693(02)00164-3}
  {\path{doi:10.1016/S0021-8693(02)00164-3}}.

\bibitem[Fac12]{facchini12}
A.~Facchini.
\newblock Direct-sum decompositions of modules with semilocal endomorphism
  rings.
\newblock {\em Bull. Math. Sci.}, 2(2):225--279, 2012.
\newblock \href {http://dx.doi.org/10.1007/s13373-012-0024-9}
  {\path{doi:10.1007/s13373-012-0024-9}}.

\bibitem[Fac19]{facchini19}
A.~Facchini.
\newblock {\em Semilocal categories and modules with semilocal endomorphism
  rings}, volume 331 of {\em Progress in Mathematics}.
\newblock Birkh\"{a}user/Springer, Cham, 2019.
\newblock \href {http://dx.doi.org/10.1007/978-3-030-23284-9}
  {\path{doi:10.1007/978-3-030-23284-9}}.

\bibitem[FGKT17]{fan-geroldinger-kainrath-tringali17}
Y.~Fan, A.~Geroldinger, F.~Kainrath, and S.~Tringali.
\newblock Arithmetic of commutative semigroups with a focus on semigroups of
  ideals and modules.
\newblock {\em J. Algebra Appl.}, 16(12):1750234, 42, 2017.
\newblock \href {http://dx.doi.org/10.1142/S0219498817502346}
  {\path{doi:10.1142/S0219498817502346}}.

\bibitem[FH00a]{facchini-herbera00}
A.~Facchini and D.~Herbera.
\newblock {$K_0$} of a semilocal ring.
\newblock {\em J. Algebra}, 225(1):47--69, 2000.
\newblock \href {http://dx.doi.org/10.1006/jabr.1999.8092}
  {\path{doi:10.1006/jabr.1999.8092}}.

\bibitem[FH00b]{facchini-herbera00a}
A.~Facchini and D.~Herbera.
\newblock Projective modules over semilocal rings.
\newblock In {\em Algebra and its applications ({A}thens, {OH}, 1999)}, volume
  259 of {\em Contemp. Math.}, pages 181--198. Amer. Math. Soc., Providence,
  RI, 2000.
\newblock \href {http://dx.doi.org/10.1090/conm/259/04094}
  {\path{doi:10.1090/conm/259/04094}}.

\bibitem[FHKW06]{facchini-hassler-wiegand06}
A.~Facchini, W.~Hassler, L.~Klingler, and R.~Wiegand.
\newblock Direct-sum decompositions over one-dimensional {C}ohen-{M}acaulay
  local rings.
\newblock In {\em Multiplicative ideal theory in commutative algebra}, pages
  153--168. Springer, New York, 2006.
\newblock \href {http://dx.doi.org/10.1007/978-0-387-36717-0_10}
  {\path{doi:10.1007/978-0-387-36717-0_10}}.

\bibitem[FHL13]{fontana-houston-lucas13}
M.~Fontana, E.~Houston, and T.~Lucas.
\newblock {\em Factoring ideals in integral domains}, volume~14 of {\em Lecture
  Notes of the Unione Matematica Italiana}.
\newblock Springer, Heidelberg, 2013.
\newblock \href {http://dx.doi.org/10.1007/978-3-642-31712-5}
  {\path{doi:10.1007/978-3-642-31712-5}}.

\bibitem[FT18]{fan-tringali18}
Y.~Fan and S.~Tringali.
\newblock Power monoids: a bridge between factorization theory and arithmetic
  combinatorics.
\newblock {\em J. Algebra}, 512:252--294, 2018.
\newblock \href {http://dx.doi.org/10.1016/j.jalgebra.2018.07.010}
  {\path{doi:10.1016/j.jalgebra.2018.07.010}}.

\bibitem[FW04]{facchini-wiegand04}
A.~Facchini and R.~Wiegand.
\newblock Direct-sum decompositions of modules with semilocal endomorphism
  rings.
\newblock {\em J. Algebra}, 274(2):689--707, 2004.
\newblock \href {http://dx.doi.org/10.1016/j.jalgebra.2003.06.004}
  {\path{doi:10.1016/j.jalgebra.2003.06.004}}.

\bibitem[Gab14]{gabelli14}
S.~Gabelli.
\newblock Ten problems on stability of domains.
\newblock In {\em Commutative algebra}, pages 175--193. Springer, New York,
  2014.

\bibitem[Ger16]{geroldinger16}
A.~Geroldinger.
\newblock Sets of lengths.
\newblock {\em Amer. Math. Monthly}, 123(10):960--988, 2016.
\newblock \href {http://dx.doi.org/10.4169/amer.math.monthly.123.10.960}
  {\path{doi:10.4169/amer.math.monthly.123.10.960}}.

\bibitem[GH08]{geroldinger-hassler-08}
A.~Geroldinger and W.~Hassler.
\newblock Local tameness of {$v$}-{N}oetherian monoids.
\newblock {\em J. Pure Appl. Algebra}, 212(6):1509--1524, 2008.
\newblock \href {http://dx.doi.org/10.1016/j.jpaa.2007.10.020}
  {\path{doi:10.1016/j.jpaa.2007.10.020}}.

\bibitem[GHK06]{geroldinger-halterkoch06}
A.~Geroldinger and F.~Halter-Koch.
\newblock {\em Non-unique factorizations}, volume 278 of {\em Pure and Applied
  Mathematics (Boca Raton)}.
\newblock Chapman \& Hall/CRC, Boca Raton, FL, 2006.
\newblock Algebraic, combinatorial and analytic theory.
\newblock \href {http://dx.doi.org/10.1201/9781420003208}
  {\path{doi:10.1201/9781420003208}}.

\bibitem[GPGS18]{gao-perezgimenez-sato18}
P.~Gao, X.~P\'{e}rez-Gim\'{e}nez, and C.~M. Sato.
\newblock Arboricity and spanning-tree packing in random graphs.
\newblock {\em Random Structures Algorithms}, 52(3):495--535, 2018.
\newblock \href {http://dx.doi.org/10.1002/rsa.20743}
  {\path{doi:10.1002/rsa.20743}}.

\bibitem[GR09]{geroldinger-rusza09}
A.~Geroldinger and I.~Z. Ruzsa.
\newblock {\em Combinatorial number theory and additive group theory}.
\newblock Advanced Courses in Mathematics. CRM Barcelona. Birkh\"{a}user
  Verlag, Basel, 2009.
\newblock Courses and seminars from the DocCourse in Combinatorics and Geometry
  held in Barcelona, 2008.
\newblock \href {http://dx.doi.org/10.1007/978-3-7643-8962-8}
  {\path{doi:10.1007/978-3-7643-8962-8}}.

\bibitem[GZ19]{geroldinger-zhong19}
A.~Geroldinger and Q.~Zhong.
\newblock Sets of arithmetical invariants in transfer {K}rull monoids.
\newblock {\em J. Pure Appl. Algebra}, 223(9):3889--3918, 2019.
\newblock \href {http://dx.doi.org/10.1016/j.jpaa.2018.12.011}
  {\path{doi:10.1016/j.jpaa.2018.12.011}}.

\bibitem[GZ20]{geroldinger-zhong20}
A.~Geroldinger and Q.~Zhong.
\newblock Factorization theory in commutative monoids.
\newblock {\em Semigroup Forum}, 100(1):22--51, 2020.
\newblock \href {http://dx.doi.org/10.1007/s00233-019-10079-0}
  {\path{doi:10.1007/s00233-019-10079-0}}.

\bibitem[HK08]{halter-koch08}
F.~Halter-Koch.
\newblock Non-unique factorizations of algebraic integers.
\newblock {\em Funct. Approx. Comment. Math.}, 39(part 1):49--60, 2008.
\newblock \href {http://dx.doi.org/10.7169/facm/1229696553}
  {\path{doi:10.7169/facm/1229696553}}.

\bibitem[HP10]{herbera-prihoda10}
D.~Herbera and P.~P\v{r}\'{\i}hoda.
\newblock Big projective modules over noetherian semilocal rings.
\newblock {\em J. Reine Angew. Math.}, 648:111--148, 2010.
\newblock \href {http://dx.doi.org/10.1515/CRELLE.2010.081}
  {\path{doi:10.1515/CRELLE.2010.081}}.

\bibitem[HS06]{huneke-swanson06}
C.~Huneke and I.~Swanson.
\newblock {\em Integral closure of ideals, rings, and modules}, volume 336 of
  {\em London Mathematical Society Lecture Note Series}.
\newblock Cambridge University Press, Cambridge, 2006.

\bibitem[LHGL14]{liu-hong-gu-lai14}
Q.~Liu, Y.~Hong, X.~Gu, and H.-J. Lai.
\newblock Note on edge-disjoint spanning trees and eigenvalues.
\newblock {\em Linear Algebra Appl.}, 458:128--133, 2014.
\newblock \href {http://dx.doi.org/10.1016/j.laa.2014.05.044}
  {\path{doi:10.1016/j.laa.2014.05.044}}.

\bibitem[LLT19]{liu-lai-tian19}
R.~Liu, H.-J. Lai, and Y.~Tian.
\newblock Spanning tree packing number and eigenvalues of graphs with given
  girth.
\newblock {\em Linear Algebra Appl.}, 578:411--424, 2019.
\newblock \href {http://dx.doi.org/10.1016/j.laa.2019.05.022}
  {\path{doi:10.1016/j.laa.2019.05.022}}.

\bibitem[LO96a]{levy-odenthal96a}
L.~S. Levy and C.~J. Odenthal.
\newblock Krull-{S}chmidt theorems in dimension {$1$}.
\newblock {\em Trans. Amer. Math. Soc.}, 348(9):3391--3455, 1996.
\newblock \href {http://dx.doi.org/10.1090/S0002-9947-96-01619-4}
  {\path{doi:10.1090/S0002-9947-96-01619-4}}.

\bibitem[LO96b]{levy-odenthal96b}
L.~S. Levy and C.~J. Odenthal.
\newblock Package deal theorems and splitting orders in dimension {$1$}.
\newblock {\em Trans. Amer. Math. Soc.}, 348(9):3457--3503, 1996.
\newblock \href {http://dx.doi.org/10.1090/S0002-9947-96-01620-0}
  {\path{doi:10.1090/S0002-9947-96-01620-0}}.

\bibitem[LW85]{levy-wiegand85}
L.~S. Levy and R.~Wiegand.
\newblock Dedekind-like behavior of rings with {$2$}-generated ideals.
\newblock {\em J. Pure Appl. Algebra}, 37(1):41--58, 1985.
\newblock \href {http://dx.doi.org/10.1016/0022-4049(85)90086-6}
  {\path{doi:10.1016/0022-4049(85)90086-6}}.

\bibitem[LW12]{leuschke-wiegand12}
G.~J. Leuschke and R.~Wiegand.
\newblock {\em Cohen-{M}acaulay representations}, volume 181 of {\em
  Mathematical Surveys and Monographs}.
\newblock American Mathematical Society, Providence, RI, 2012.
\newblock \href {http://dx.doi.org/10.1090/surv/181}
  {\path{doi:10.1090/surv/181}}.

\bibitem[Mat73]{matlis73}
E.~Matlis.
\newblock {\em {$1$}-dimensional {C}ohen-{M}acaulay rings}.
\newblock Lecture Notes in Mathematics, Vol. 327. Springer-Verlag, Berlin-New
  York, 1973.

\bibitem[Mat89]{matsumura89}
H.~Matsumura.
\newblock {\em Commutative ring theory}, volume~8 of {\em Cambridge Studies in
  Advanced Mathematics}.
\newblock Cambridge University Press, Cambridge, second edition, 1989.
\newblock Translated from the Japanese by M. Reid.

\bibitem[Olb16]{Olberding-16}
B.~Olberding.
\newblock One-dimensional stable rings.
\newblock {\em J. Algebra}, 456:93--122, 2016.
\newblock \href {http://dx.doi.org/10.1016/j.jalgebra.2016.02.002}
  {\path{doi:10.1016/j.jalgebra.2016.02.002}}.

\bibitem[Pal01]{palmer01}
E.~M. Palmer.
\newblock On the spanning tree packing number of a graph: a survey.
\newblock {\em Discrete Math.}, 230(1-3):13--21, 2001.
\newblock Paul Catlin memorial collection (Kalamazoo, MI, 1996).
\newblock \href {http://dx.doi.org/10.1016/S0012-365X(00)00066-2}
  {\path{doi:10.1016/S0012-365X(00)00066-2}}.

\bibitem[Sch16]{schmid16}
W.~A. Schmid.
\newblock Some recent results and open problems on sets of lengths of {K}rull
  monoids with finite class group.
\newblock In {\em Multiplicative ideal theory and factorization theory}, volume
  170 of {\em Springer Proc. Math. Stat.}, pages 323--352. Springer, [Cham],
  2016.
\newblock \href {http://dx.doi.org/10.1007/978-3-319-38855-7_14}
  {\path{doi:10.1007/978-3-319-38855-7_14}}.

\bibitem[Tri18]{tringali18}
S.~Tringali.
\newblock Structural properties of subadditive sequences with applications to
  factorization theory and additive combinatorics.
\newblock Preprint (Accessed 2020-06-12), 2018.
\newblock URL: \url{https://imsc.uni-graz.at/tringali/docs/sub.pdf}.

\bibitem[Tri19]{tringali19}
S.~Tringali.
\newblock Structural properties of subadditive families with applications to
  factorization theory.
\newblock {\em Israel J. Math.}, 234(1):1--35, 2019.
\newblock \href {http://dx.doi.org/10.1007/s11856-019-1922-2}
  {\path{doi:10.1007/s11856-019-1922-2}}.

\bibitem[Wie84]{wiegand-84}
R.~Wiegand.
\newblock Cancellation over commutative rings of dimension one and two.
\newblock {\em J. Algebra}, 88(2):438--459, 1984.
\newblock \href {http://dx.doi.org/10.1016/0021-8693(84)90077-2}
  {\path{doi:10.1016/0021-8693(84)90077-2}}.

\bibitem[Wie01]{wiegand01}
R.~Wiegand.
\newblock Direct-sum decompositions over local rings.
\newblock {\em J. Algebra}, 240(1):83--97, 2001.
\newblock \href {http://dx.doi.org/10.1006/jabr.2000.8657}
  {\path{doi:10.1006/jabr.2000.8657}}.

\bibitem[WW09]{wiegand-wiegand09}
R.~Wiegand and S.~Wiegand.
\newblock Semigroups of modules: a survey.
\newblock In {\em Rings, modules and representations}, volume 480 of {\em
  Contemp. Math.}, pages 335--349. Amer. Math. Soc., Providence, RI, 2009.
\newblock \href {http://dx.doi.org/10.1090/conm/480/09384}
  {\path{doi:10.1090/conm/480/09384}}.

\bibitem[Zho19]{zhong19}
Q.~Zhong.
\newblock On elasticities of locally finitely generated monoids.
\newblock {\em J. Algebra}, 534:145--167, 2019.
\newblock \href {http://dx.doi.org/10.1016/j.jalgebra.2019.05.031}
  {\path{doi:10.1016/j.jalgebra.2019.05.031}}.

\end{thebibliography}

%\newpage

%\section*{Declarations}

%\noindent\textbf{Funding} Smertnig was supported by the Austrian Science Fund (FWF) project J4079-N32. Part of the research was conducted while Smertnig was visiting the University of Waterloo.

%\noindent\textbf{Conflicts of interest/Competing interests} None.

%\noindent\textbf{Availability of data and material} Data sharing not applicable to this article as no datasets were generated or analysed during the current study.

%\noindent\textbf{Code availability} Code sharing not applicable to this article as no code was generated during the current study.

\end{document}